\documentclass[12pt,a4paper]{article}

\usepackage{graphicx}
\usepackage{latexsym,amssymb,amsmath}
\usepackage{amsfonts, mathrsfs}
\usepackage{latexsym}
\usepackage{graphicx}
\usepackage{color}
\usepackage{amsthm}
\usepackage{enumitem}
\usepackage{tikz}
\usepackage[all]{xy}

\theoremstyle{plain}
\newtheorem{theorem}{Theorem}[section]
\newtheorem{proposition}{Proposition}[section]
\newtheorem{corollary}{Corollary}[section]

\theoremstyle{definition}
\newtheorem{definition}{Definition}[section]
\newtheorem{remark}{Remark}[section]

\numberwithin{equation}{section}
\normalsize

\makeatletter
\renewcommand\section{\@startsection{section}{1}{\z@}%
                                   {-5.5ex \@plus -1ex \@minus -.2ex}%
                                   {1.3ex \@plus.2ex}%
                                   {\normalfont\normalsize\bfseries}}
\renewcommand{\@seccntformat}[1]{\csname the#1\endcsname.\ }

\renewcommand\subsection{\@startsection{subsection}{2}{\z@}%
                                   {-2.5ex \@plus -1ex \@minus -.2ex}%
                                   {1.3ex \@plus.2ex}%
                                   {\normalfont\normalsize\it}}
\makeatother

\DeclareMathOperator{\epi}{epi}

\DeclareMathOperator{\dom}{dom}
\DeclareMathOperator{\inte}{int}
\DeclareMathOperator{\co}{co}
\DeclareMathOperator{\cone}{cone}

\DeclareMathOperator{\supp}{supp}

\DeclareMathOperator{\gph}{gph}

\DeclareMathOperator{\oco}{\overline{\co}}

\def\RR{\mathbb{R}}
\def\A{\mathcal{A}}

\def\F{\mathcal{F}}
\def\L{\mathcal{L}}
\def\U{\mathcal{U}}

\def\R{\mathbb{R}}

\def\la{\langle}
\def\ra{\rangle}

\def\ms{\medskip }

\def\barV0{\widetilde{\mathcal V}_0 } 

\allowdisplaybreaks
\begin{document}

\setlength{\abovedisplayskip}{3pt}
\setlength{\belowdisplayskip}{3pt}

\date{} 

\title{\large\bf  Duality for Robust Linear Infinite Programming Problems Revisited\footnote{This work is partly supported by the project 101.01-2018.310, NAFOSTED, Vietnam }}

\author{{ \normalsize N.~Dinh}\thanks{
International University, Vietnam National University - HCMC, Linh Trung
ward, Thu Duc district, Ho Chi Minh city, Vietnam ({\it ndinh@hcmiu.edu.vn}). Part of  the work  of this author was realized when he visited Center for General Education, China Medical University,  Taiwan. He expresses his sincere thanks  to the hospitality he  received.     } 
\and {\normalsize D.H.~Long}\thanks{VNUHCM - University of Science, District 5, Ho Chi Minh city, Vietnam, and  Tien Giang University, Tien Giang town, Vietnam  ({\it danghailong@tgu.edu.vn})}\and  {\normalsize J.-C. Yao}\thanks{Center for General Education, China Medical University, Taichung 40402, Taiwan ({\it yaojc@mail.cmu.edu.tw})}}

\maketitle

\centerline{\it Dedicated to Professor Marco Antonio L\' opez's 70$^{th}$ birthday}

\begin{abstract} In this paper, we consider  the robust linear infinite programming  problem $({\rm RLIP}_c) $ defined by   
\begin{eqnarray*}
({\rm RLIP}_c)\quad &&\inf\;  \langle c,x\rangle \\
\textrm{subject to } &&x\in X,\; \langle x^\ast,x \rangle \le r ,\;\forall (x^\ast,r)\in\mathcal{U}_t,\; \forall t\in T, 
\end{eqnarray*}
where $X$ is a locally convex Hausdorff topological vector space,  $T$ is an arbitrary (possible infinite) index set, $c\in X^*$,   and $\mathcal{U}_t\subset X^*\times \mathbb{R}$, $t \in T$ are    uncertainty sets.     

We  propose an  approach to duality for the robust linear problems with convex constraints  $({\rm RP}_c)$  and establish corresponding   robust strong duality and also, stable robust strong duality, i.e., robust strong duality holds "uniformly" with all $c \in X^\ast$. 
 With the different choices/ways of setting/arranging data   from  $({\rm RLIP}_c) $, one gets  back to the model  $({\rm RP}_c)$ and the (stable) robust strong duality for   $({\rm RP}_c)$ applies. By such a way, nine versions of   dual problems for $ ({\rm RLIP}_c)$  are proposed. Necessary and sufficient conditions for stable robust strong duality of these pairs of primal-dual problems are given, for which some  cover several known  results  in the literature while the others,  due to the best knowledge of the authors, are new. Moreover, as by-products, we  obtained  from the robust strong duality for variants pairs of primal-dual problems, several robust Farkas-type results for  linear infinite systems with uncertainty. 
 Lastly, as extensions/applications, we extend/apply  the results obtained to robust linear problems with sub-affine constraints, and to linear infinite problems (i.e., $({\rm RLIP}_c) $ with the absence of uncertainty). 
It is worth noticing even for these  cases, we are able to derive  new results on (robust/stable robust) duality for the mentioned classes of problems and new robust   Farkas-type results for  sub-linear systems, and also  for linear infinite systems  in the absence of  uncertainty.  
\end{abstract}



\section{Introduction }

In this paper, we are concerned with the {\it linear infinite programming with uncertainty parameters}  of the form
 \begin{eqnarray*}
({\rm LIP}_c)\quad &&\inf\;  \langle c,x\rangle \\
\textrm{subject to} &&x\in X,\;\langle a_t,x \rangle \le b_t,\;\forall t\in T,
\end{eqnarray*}
where $X$ is a locally convex Hausdorff topological vector space,  $T$ is an arbitrary (possible infinite) index set, $c\in X^*$, $a_t \in  X^*$ and $b_t \in \mathbb{R}$   for each $t \in T$,  and the couple $(a_t , b_t )$ belongs to an uncertainty set $\mathcal{U}_t\subset X^*\times \mathbb{R}$. For such a  {\it linear infinite programming} ${\rm (LIP}_c)$ with {\it input-parameter uncertainty},  its robust  counterpart   is the robust linear infinite programing problem $({\rm RLIP}_c) $ defined as follows:   
\begin{eqnarray*}
({\rm RLIP}_c)\quad &&\inf\;  \langle c,x\rangle \\
\textrm{subject to  } &&x\in X,\; \langle x^\ast,x \rangle \le r ,\;\forall (x^\ast,r)\in\mathcal{U}_t,\; \forall t\in T.
\end{eqnarray*}

The   robust linear infinite  problems of the model  $({\rm RLIP}_c) $ together with their duality  were    considered in several works in the literature  such as, \cite{ChJeya17}, \cite{DGLM17-Optim}, \cite{DMVV17-Robust-SIAM}, \cite{GJLL},      \cite{JL11}.  There are   variants of    duality results for robust  convex problems (see  \cite{Bot}, \cite {CLLLY},  \cite{DGV-AMO}, \cite {DGV19a},     \cite{DGLV17-Robust}, \cite{FLY15},   \cite{DMVV17-Robust-SIAM},  \cite{JL2},       \cite {LJL11}  and references therein), and also for robust  vector optimization/multi-objective  problems (see, e.g., \cite{CKY19}, \cite{DGLM17-Optim}, \cite {DL18ACTA}, \cite{GJLb}). In the mentioned papers,  results for robust strong   duality are established for classes of problems from linear to convex, non-convex,  and vector problems, under various constraint qualification conditions  (or qualification conditions).

In this paper  we propose a way,  which can be considered as a unification approach to duality for the robust linear problems $({\rm RLIP}_c) $. Concretely, we propose some  model for a bit more general problem, namely, the robust linear problem with convex conical constraints  $({\rm RP}_c)$    and establish corresponding robust strong duality and also, stable robust strong duality, i.e., robust strong duality holds "uniformly" with all $c \in X^\ast$. 
Then, with the different choices/ways  of setting,   we transfer  $({\rm RLIP}_c) $ to the models  $({\rm RP}_c)$,  and the (stable) robust strong duality  results for   $({\rm RP}_c)$ apply.  
By such a way, several forms of   dual problems for $ ({\rm RLIP}_c)$  are proposed. Necessary and sufficient conditions for stable robust strong duality of these pairs of primal-dual problems are given, for  which some cover  results known in the literature while the others,  due to the best knowledge of the authors, are new. 
We point out also  that,  even  in the case with the  absence of uncertainty, i.e., in the case where $ \mathcal{U}_t $ is singleton for each $t \in T$, the results obtained still lead to new results on duality for robust linear infinite/semi-infinite  problems (see Section 6). 

The paper is organized as follows: In Section 2, some preliminaries and basic tools are introduced. Concretely, we introduce or quote some robust Farkas lemmas for conical constraint systems under uncertainty, some results on duality of robust linear problems with convex conical constraints. The model of robust linear infinite problem and its  seven   models of  robust dual problems are given in Section 3. 
The main results: Robust stable strong duality results for  $ ({\rm RLIP}_c)$   are given in Section 4 together with two more models of robust dual problems  of $ ({\rm RLIP}_c)$. Here, the  stable strong duality for the seven pairs of primal-dual problems are established and the ones for  two new  pairs are mentioned. Some of these results cover or extend some in \cite{DGLV17-Robust}, \cite{GJLL}.  In Section 5, from  the duality results in Section 4, we derive  variants  of stable robust  Farkas lemmas for linear infinite systems with uncertainty which cover the ones in   \cite{DGLM17-Optim}, \cite{DMVV17-Robust-SIAM}  while the others are new. In  Section 6, as an extension/application of the approach, we   get robust strong duality results for linear problems with sub-affine constraints. We consider a particular case  with   the absence of uncertainty, (i.e., in the case where $ \mathcal{U}_t $ is singleton for each $t \in T$) the results obtained still lead to some new results on duality for robust linear infinite/semi-infinite  problems, and, in turn, these results also give rise to several new versions of Farkas lemmas for  sub-affine systems under uncertainty and also, some new versions  of Farkas-type results   for linear infinite/semi-infinite systems. 

\section{Preliminaries and  Basic Tools}

Let $X$ and $Z$ be locally  convex  Hausdorff  topological vector spaces  with topological dual spaces $X^{\ast }$ and $Z^{\ast }$, respectively. The only topology considered  on
dual spaces is the weak*-topology. Let  $S$ be a  non-empty  closed and convex cone in  $Z$. The positive dual cone $S^+$ of  $S$ is  $S^+:= \{ z^* \in Z^\ast:  \la z^*, s\ra \geq 0, \forall s \in S\}$.   Let further,  $\Gamma (X)$ be the set of all proper, convex and lower semi-continuous (briefly, lsc) functions on $X$.   
Denote by $\mathcal{L}(X,Z)$ the space of all continuous linear mappings  from $X$ to $Z$ and 
 $ \overline{\mathbb{R}}:=\mathbb{R}\cup\{\pm\infty\}$,  $\mathbb{R}_{\infty}:=\mathbb{R}\cup\{+\infty\}$.

 \subsection{Notations and prelimaries}   
    
 We now give some notations which will be used in the sequent. 
    For $f\colon X\to \overline{\mathbb{R}}:=\mathbb{R}\cup\{\pm\infty\}$, the domain and the epigraph of $f$ are defined respectively by
\begin{align*}
\dom f&:=\{x\in X:\ f(x)\neq +\infty \},\\
\epi f&:=\{(x,r)\in X\times \mathbb{R}: 
f(x)\leq r\},
\end{align*}%
while its conjugate function   $f^\ast\colon X\rightarrow \overline{\mathbb{R}}$\  is  
\begin{equation*}
f^{\ast }\left( x^{\ast }\right) :=\sup_{x\in X}\left[ \left\langle x^{\ast
},x\right\rangle -f\left( x\right) \right] ,\quad \forall x^{\ast }\in X^{\ast }.
\end{equation*}%

Let $ \leq_{S}$ be  the
ordering on $Z$\ induced by the cone $S$, i.e., 
\begin{equation}  \label{order}
z_{1}\leqq _{S}z_{2}\ \text{ if and only if }\ z_{2}-z_{1}\in S.
\end{equation}%
We enlarge $Z$ by attaching a greatest element $+\infty _{Z}$ and a
smallest element $-\infty _{Z}$  which do not
belong to $Z$ by the  convention, $-\infty _{Z} \leqq _{S} z  \leqq_{S}   +\infty_{Z}$ for all $z \in Z$. 
Denote  $Z^{\bullet }:=Z\cup \{-\infty _{Z},\ +\infty
_{Z}\}$. Let $G\colon X\to Z^\bullet$.  We define
\begin{align*}
\dom G&:=\{x\in X:\ G(x)\neq +\infty_Z \},\\
\epi_S G&:=\{(x,z)\in X\times Z: 
z \in G(x) + S \}.
\end{align*}%
If $-\infty_Z\notin G(X)$ and $\dom G\ne\emptyset$, then we say that $G$ is a {\it proper mapping}. We say that $G$ is $S$-\emph{convex} (resp., $S$-\emph{epi closed}) if $%
\epi_{S}G $ is a convex subset (resp., a closed subset) of
$X\times Z$. The mapping $G $ is called \emph{positively $S$-upper
semicontinuous}\footnote{In \cite{Bot10} this notion is named as  {\it Star $S$-usc}}  (\emph{positively $S$-usc}, briefly) if $\lambda G$ is  upper semicontinuous (in short, usc) for all $\lambda\in S^{+}$ (see  \cite{Bol98}, \cite{Bol01}).

Let $T$ be an index   (possibly infinite)  set and let $\mathbb{R}^{T}$ be  the product space   endowed with the
product topology and its dual space, $\mathbb{R}^{(T)}$, the so-called \emph{space of generalized finite sequences} $%
\lambda =(\lambda _{t})_{t\in T}$ such that $\lambda _{t}\in \mathbb{R},$
for each $t\in T,$ and with only finitely many $\lambda _{t}$ different from
zero. The supporting set of $\lambda \in \mathbb{R}^{(T)}$ is $\supp
\lambda :=\{t\in T \, :\  \lambda _{t}\neq 0\}.$ 
 For a pair   $\left( \lambda ,v\right) \in \mathbb{R}^{(T)}\times \mathbb{R
}^{T}$,  
  the dual
product  is defined by 
\begin{equation*}
\left\langle \lambda ,v\right\rangle :=\left\{ 
\begin{array}{ll}
\sum\limits_{t\in \supp\lambda }\lambda _{t}v_{t} & \text{if }\lambda
\neq 0_{T}, \\ 
\ \ 0 & \text{otherwise}. %
\end{array}%
\right. \text{ }
\end{equation*}%
 The positive cones in  $\mathbb{
R}^{T}$ and in $\mathbb{%
R}^{(T)}$ are denoted by $\mathbb{R}_+^T$ and  $\mathbb{R}_{+}^{(T)}$, respectively. 

\ms

  {\it $S^+$-Upper Semi-Continuity  and Uniform $S^+$-Concavity}.  
Let  $U \ne \emptyset$ be a  subset of some topological  space. We recall the notions of $S^+$-upper semi-continuity, $S^+$-convexity,   and uniform $S^+$-convexity  introduced recently in \cite{DL18ACTA}.    

\begin{definition}\label{def_1}
\cite{DL18ACTA}
Let $H\colon U\to Z\cup\{+\infty_Z\}$. We say that:

$\bullet$ $H$ is {\it $S^+$-convex} if for all  $(u_i,\lambda_i)\in U\times S^+ \; (i=1,2)$ there is $(\bar u,\bar \lambda)\in U\times S^+$ such that $(\lambda_1 H)(u_1)+(\lambda_2 H)(u_2) \ge (\bar \lambda H)(\bar u)$,

$\bullet$ $H$ is {\it $S^+$-upper semi-continuos} (briefly, {\it $S^+$-usc})   if for any net $(\lambda_\alpha, u_\alpha,r_\alpha)_{\alpha\in D}\subset S^+\times U\times \RR$ and any  $(\bar \lambda,\bar u,\bar r)\in S^+\times U\times \RR$, satisfying 
$$ \begin{cases}
(\lambda_\alpha H)( {u_\alpha}) \ge r_\alpha,\; \forall \alpha \in D\\
\lambda_\alpha \overset{*}{\rightharpoonup} \bar \lambda,\; u_\alpha\to \bar u,\; r_\alpha\to\bar r
\end{cases}
\Longrightarrow\; (\bar \lambda H)(\bar u)\ge \bar r, 
$$
where the  symbol ``$\overset{*}{\rightharpoonup}$" means the convergence with respect to  weak$^*$-topology.  

 $\bullet$ $H$ is {\it $S^+$-concave} ({\it $S^+$-lsc}, resp.) if $-H$ is {\it $S^+$-convex} ({\it $S^+$-usc}, respectively).
\end{definition}

\begin{definition}\label{def_2}
\cite{DL18ACTA}  
For  the collection $(H_j)_{j \in I}$ with $H_j\colon  U\to Z\cup\{+\infty_Z\}$,  we say that $(H_j)_{j\in I}$
is {\it uniformly $S^+$-convex} if  for all  $(u_i,\lambda_i)\in U\times S^+, \; i=1,2$,  there is $(\bar u,\bar \lambda)\in U\times S^+$ such that $(\lambda_1 H_j)(u_1)+ (\lambda_2 H_j)(u_2) \ge (\bar \lambda H_j)(\bar u)$ for all $j\in I$. 

The collection  $(H_j)_{j\in I}$ is said to be {\it uniformly $S^+$-concave} if  $(-H_j)_{j\in I}$ is   uniformly $S^+$-convex.
\end{definition}
\begin{remark} 
\label{rem_2.1eeee} It is worth observing that when $H\colon   U \to Z\cup \{+\infty_Z\}$ is   {$S^+$-usc} then $G$ is  positively $S$-usc \cite{Bot10}. Moreover, 
in  the case where $Z=\mathbb{R}$ and $S=\mathbb{R}_+$, (and hence, $S^+=\mathbb{R}_+$),    the following assertions hold\footnote{For a  function, we prefer the the lowercase letter $h$  to  $H$.}: 

{\rm (i)}  If $H\colon U\to \mathbb{R}_{\infty}$ is a convex function then $H$ is $\mathbb{R}_+$-convex, 

{\rm (ii)} 
If $H_j\colon  U \to \mathbb{R}_\infty$ is convex for all $j\in I$ then   $(H_j)_{j\in I}$ is {uniformly $\mathbb{R}_+$-convex}. 

{\rm (iii)}   $H\colon   U \to  \mathbb{R}_\infty$ is  {$\RR_+$-usc} if and only if it is usc.

\noindent For details, see \cite{DL18ACTA}. 
\end{remark}


\subsection{Conical Constrained Systems with Uncertainty}  

Let  $\U$ be an   uncertainty parameter set, $(G_u)_{u\in \U}$ with $G_u\colon X\to Z\cup\{+\infty_Z\}$,  be a proper $S$-convex and $S$-epi closed mapping for each $u \in \U$. 
 We are concerned with the robust cone constraint system:
 \begin{equation}
\label{eq_3.1bbb}
G_u(x)\in-S,\quad \forall u\in \U.
\end{equation}
Denote 
\begin{equation} \label{Fu}
\mathcal{F}_u:=\{x\in X: G_u(x)\in-S\},\   u\in \U, 
\end{equation} 
   and   
 $ \mathcal{F}$ the \textit{solution set} of \eqref{eq_3.1bbb}, i.e., 
\begin{equation} \label{F}
\mathcal{F}:=\{x\in X: G_u(x)\in -S,\; \forall u\in \U\}.
\end{equation} 
It is clear that $\mathcal{F} = \bigcap \limits_{u \in \U} \mathcal{F}_u$. 
 Assume that $\mathcal{F}\ne\emptyset$. 
 
Corresponding to  the  system \eqref{eq_3.1bbb}, let us consider the set (also called: {\it robust moment cone} corresponding to the system \eqref{eq_3.1bbb})  
\begin{equation}\label{M0}
\mathcal{M}_0:=\bigcup_{(u,\lambda)\in \U\times S^+} \epi (\lambda G_u)^\ast.
\end{equation} 
It is easy to check that (generalizing the one in \cite[Proposition 2.2]{JL2}) $\mathcal{M}_0$ is a cone in $X^\ast\times \mathbb{R}$.  Moreover,  $\mathcal{M}_0$ (and also $\mathcal{M}_1$ in \eqref{M_1}) leads to the cone $M$ in \cite{GJLL}.

We now introduce a version of robust Farkas-type result involving the system \eqref{eq_3.1bbb}  and some of its consequences, which will be used as a key tool for the  results of this section.

\begin{proposition}
[Farkas-type result involving  robust system \eqref{eq_3.1bbb}]
\label{lem_FL}
For all $(x^\ast, r)\in X^\ast\times\mathbb{R}$, the next statements are equivalent:

{\rm (i)} $G_u(x)\in -S,\ \forall u\in \U \; \Longrightarrow\; \la x^\ast, x\ra \ge r$,

{\rm (ii)} $(x^\ast, r)\in -\oco \mathcal{M}_0$.
\end{proposition}

\begin{proof}
It is easy to see that (i) is equivalent to $ -r\ge -\la x^\ast, x\ra$ for all $x\in \mathcal{F}$ ,
which also means $(x^\ast, r)\in -\epi \delta_\mathcal{F}^\ast.$
So, to prove the equivalence  (i)$\Longleftrightarrow$(ii), it suffices to show  that  $\epi \delta_\mathcal{F}^\ast=\oco\mathcal{M}_0$.

Now,  for each  $u\in \U$, $\mathcal{F}_u$ is closed and convex subsets of $X$, and hence, $\delta_{\mathcal{F}_u} \in \Gamma(X)$ and  so $\delta_{\mathcal{F}}=\sup_{u\in \U}\delta_{\mathcal{F}_u} \in \Gamma (X)$.  By  \cite[Lemma 2.2]{GNg08}, one gets
$\epi \delta_{\mathcal{F}}^\ast=\oco \bigcup_{u\in \U}\epi \delta_{\mathcal{F}_u}^\ast.$
On the other hand, for each  $u\in \U$,  one  has  $\epi \delta_{\mathcal{F}_u}^\ast=\overline{\bigcup_{\lambda\in S^+} \epi(\lambda G_u)^\ast}$ (see \cite{DNV10}), and so, 
  $\epi \delta_{\mathcal{F}}^\ast=\oco\mathcal{M}_0$ and we are done.
\end{proof}

As a direct consequence of Proposition \ref{lem_FL}, we get
\begin{corollary}
\label{lem_FL2}
Let $(A_u)_{u\in \U}\subset \L(X,Z)$ and $(\omega_u)_{u\in \U}\subset Z$.  If  $(x^\ast, s)\in X^\ast\times\mathbb{R}$ then the  next statements are equivalent:

{\rm (i)} $ A_u(x) -\omega_u\in -S,\ \forall u\in \U \;   \Longrightarrow\; \la x^\ast, x\ra \ge s$,

{\rm (ii)} $(x^\ast, s)\in -\oco \Big[\big\{(\lambda A_t, \la \lambda, \omega_t\ra) : (u,\lambda)\in \U\times S^+\big\}+\mathbb{R}_+(0_{X^\ast},1)\Big]$.
\end{corollary}

\begin{remark} 
 Corollary \ref{lem_FL2} covers \cite[Theorem 3.1]{JL11},  \cite[Theorem 4.2(iii)]{DGLM17-Optim},  \cite[Theorem 5.5]{DMVV17-Robust-SIAM}, and in some sense,  it extend the robust semi-infinite Frakas' lemmas in \cite{GJLL}.   In the case when  $Z=\mathbb{R}$ and $\U=T$, $X = \mathbb{R}^n$,  
Corollary \ref{lem_FL2}  extends \cite[Corollary 3.1.2]{GL98}. 
\end{remark}

Let  $\emptyset \ne B \subset X^*$ and $\beta \in \R$.  The function  $\sigma_B (\cdot) - \beta$, where  $\sigma_{B}(x):=\sup\{\la x^\ast, x\ra: x^\ast\in B\}$,  is known as a  sub-affine function   \cite{DGV19a}. We next give a version of robust Farkas lemma for a system involving  sub-affine  functions.

\begin{corollary}
\label{lem_FL3}
Let $( \A_t) _{t\in T}$ be a family of nonempty, $w^{\ast }$-
closed convex subsets of $X^{\ast }$ and  $(b_t)_{t\in T}\subset \mathbb{R}$. Then, for each $(x^\ast, r)\in X^\ast\times\mathbb{R}$, the next statements are equivalent:

{\rm (i)} $\sigma_{{\A}_t}( x)\le b_t,\ \forall t\in T\; \Longrightarrow\; \la x^\ast, x\ra \ge r$,

{\rm (ii)} $(x^\ast, r)\in -\oco \cone \Big[\bigcup_{t\in T}(\A_t\times\{ b_t\}) \cup \{(0_{X^*},1)\}\Big]$.
\end{corollary}

\begin{proof}
Take $Z=\mathbb{R}$, $S=\mathbb{R}_+$ (and hence, $Z^\ast=\mathbb{R}$ and $S^+=\mathbb{R}^+$), $\U=T$, and $G_t:=\sigma_{\A_t}-b_t$ for each  $t\in T$. Then, for any $(t,\lambda)\in T\times\mathbb{R}_+$, one has
\begin{align*}
\epi (\lambda G_t)^\ast&=\lambda\epi ( G_t)^\ast=\lambda\epi (\sigma_{\A_t}-b_t)^\ast=\lambda (\A_t\times\{b_t\}) +\mathbb{R}_+(0_{X^\ast},1), \\
\mathcal{M}_0
&=\bigcup_{t\in T}\co\cone\Big[(\A_t\times \{b_t\})\cup \{(0_{X^\ast},1)\}\Big],
\end{align*} 
and so, $ \oco \mathcal{M}_0=\oco\cone\Big[\bigcup_{t\in T}(\A_t\times \{b_t\})\cup \{(0_{X^\ast},1)\}\Big].$ 
The conclusion now follows from Proposition \ref{lem_FL}.
\end{proof}

\subsection{Duality of Robust Linear Problems with Convex Conical constraints} 

Let  $c\in X^\ast$. We consider the pair of primal-dual  robust problems:
\begin{align*}
	({\rm RP}_c)\qquad& {\inf  }\;  \la c,x\ra\\
\textrm{subject to}\ \  & x\in X,\; G_u(x)\in -S, \ \forall u\in \U,\\
	({\rm RD}_c)\qquad& \sup_{(u,\lambda)\in \U\times S^+}  \inf_{x\in X} (\la c,x\ra+\lambda G_u(x)).
\end{align*}


Let  $\F_u$ and $\F$ be as in \eqref{Fu} and \eqref{F}.  Let further 
   $\bar x\in \F$ and $(\bar u,\bar \lambda)\in \U\times S^+$.  As $\bar x\in \F$, $G_u(\bar x) \in-S$ for all $u\in \U$, and in particular, $G_{\bar u}(\bar x) \in-S$. Moreover,   as  $\bar \lambda\in S^+$, one has  $\bar \lambda G_{\bar u} (\bar x)\leq 0$. Therefore, 
  $
    \la c, \bar x \ra  +      (\bar \lambda G_{\bar u}) (x)    \le \la c,\ \bar x\ra , 
   $
    and so, 
 \begin{eqnarray*}
 \inf\limits_{x \in X} \Big[ \la c,x\ra  +(\bar \lambda G_{\bar u}) (x)\Big] &\le& \la c,\bar  x\ra   + (\bar \lambda G_{\bar u}) (x) 
  \leq      \la c,\bar  x\ra , 
 \end{eqnarray*} 
 leading to 
 \begin{eqnarray*}
 \inf\limits_{x \in X} \Big[ \la c,x\ra  +(\bar \lambda G_{\bar u}) (x)\Big] &\le& \inf\limits_{\bar x \in A}    \la c,\bar  x\ra. 
 \end{eqnarray*}  
 Consequently, 
 \begin{eqnarray}\label{eq_2.2ff}
\sup\limits_{(\bar u, \bar \lambda) \in \U \times S^+}   \inf\limits_{x \in X} \Big[ \la c,x\ra  +(\bar \lambda G_{\bar u}) (x)\Big] \leq 
\inf\limits_{\bar x \in A}    \la c,\bar  x\ra,   
 \end{eqnarray}   
 which means that the {\it weak duality} holds for the pair $({\rm RP}_c) -  ({\rm RD}_c)$.  
 
 \begin{definition} \label{robust-dual}
We say  that
\begin{itemize}  
\item  {\it the robust strong  duality holds for the pair $({\rm RP}_c)-({\rm RD}_c)$} if\\ $\inf ({\rm RP}_c)=\max ({\rm RD}_c)$,
\item   {\it the stable  robust strong  duality holds for the pair $({\rm RP}_c)-({\rm RD}_c)$} if\\ $\inf ({\rm RP}_c)=\max ({\rm RD}_c)$ for all $c\in X^*$.
\end{itemize}
\end{definition}  
The next theorem,   Theorem \ref{thm_StrD},  can be derived from \cite[Theorem 6.3]{DMVV17-Robust-SIAM}. However, for the sake of convenience  we will give here a short and direct   proof. 

\begin{theorem}[Characterization of stable robust strong duality for $({\rm RP}_c)$]
\label{thm_StrD}  Assume that $r_0:= \inf  ({\rm RP}_c) > - \infty$. Then 
 the following statements are equivalent:

$\rm(a)$ $\mathcal{M}_0 $ is a closed and convex subset  of $X^*\times \mathbb{R}$, 
 
$\rm(b)$  The  stable robust strong duality holds for the pair {$({\rm RP}_c)-({\rm RD}_c)$}, i.e.,  
$$\inf ({\rm RP}_c)=\max ({\rm RD}_c),\quad\forall c\in X^\ast.$$
\end{theorem}

\begin{proof}  Take arbitrarily $c\in X^\ast$.    Observe firstly that 
\begin{align}
\sup ({\rm RD}_c)&=\sup_{(u,\lambda)\in \U\times S^+} \inf\limits_{x \in X}  \Big\{ \la c, x\ra   +  (\lambda G_u)(x)\Big\}  \notag\\
 &=\sup_{(u,\lambda)\in \U\times S^+} \Big[ -  \sup \limits_{x \in X}  \Big\{ \la -c, x\ra   -  (\lambda G_u)(x)\Big\}  
 =\sup_{(u,\lambda)\in \U\times S^+}    [-(\lambda G_u)^\ast(-c)] \notag\\
&=\sup\Big\{r: (c,r)\in -\!\!\!  \bigcup\limits_{ (u,\lambda)\in\U\times S^+}   \!\!\! \gph(\lambda G_u)^\ast  \Big\}
\notag\\
&=\sup\Big\{r: (c,r)\in -\bigcup\limits_{ (u,\lambda)\in\U\times S^+}\gph(\lambda G_u)^\ast  -\mathbb{R}_+(0_{X^\ast},1)\Big\}
\notag\\
&= \sup\Big\{r: (c,r)\in -\bigcup\limits_{ (u,\lambda)\in\U\times S^+}\Big[ \gph(\lambda G_u)^\ast + \mathbb{R}_+(0_{X^\ast},1)\Big]\Big\}
\notag\\
&=  \sup \Big\{r: (c,r)\in  - \!\!\! \bigcup\limits_{ (u,\lambda)\in\U\times S^+} \!\!\!\!    \epi(\lambda G_u)^\ast  \Big\} 
=\sup  \Big\{r: (c,r)\in -\mathcal{M}_0\Big\}.
\label{eq_3.3bbb}
\end{align}
Observe also that $ r_0  < + \infty$ as $({\rm RP}_c)$ is feasible (i.e., its feasible set $\F$ is non-empty) and so, we can assume that $r_0 \in \RR$.    

$\bullet$  [(a)$\Longrightarrow$(b)]  Assume that (a) holds. 
 As $r_0= \inf  ({\rm RP}_c)$, one has 
\begin{equation}
\label{eq_3.4bbb} 
G_u\in -S,\forall u\in \U\; \Longrightarrow \la c,x\ra \ge r_0.
\end{equation}
As (a) holds, it  follows  from Proposition \ref{lem_FL} that  
$$(c,r_0)\in -\oco\mathcal{M}_0=  - \mathcal{M}_0 =\  - \bigcup_{(u,\lambda)\in \U\times S^+} \epi (\lambda G_u)^\ast, $$  
and so,  by \eqref{eq_3.3bbb} and the weak duality \eqref{eq_2.2ff}, we get 
$$ r_0 \leq \sup \{ r: (c, r) \in -\mathcal{M}_0\} = \sup ({\rm RD}_c) \leq r_0 = \inf (\rm{RP}_c),   $$ 
yielding $r_0 = \sup \{ r: (c, r) \in -\mathcal{M}_0\} = \sup ({\rm RD}_c)  = \inf (\rm{RP}_c)$.   
As 
 $r_0 \in  \{ r: (c, r) \in -\mathcal{M}_0\} $ there exist $(\bar u, \bar \lambda) \in \U \times S^+$ satisfying (by \eqref{eq_3.3bbb})  
  $$r_0 = -(\bar \lambda G_{\bar u})^\ast(-c) = \max ({\rm RD}_c)  = \inf (\rm{RP}_c), $$   
which means that (b) holds.

$\bullet$  [(b)$\Longrightarrow $(a)]  Assume that (b) holds.  To prove (a),  it suffices to show that $\oco\mathcal{M}_0\subset \mathcal{M}_0$. Take $(c,r)\in -\oco \mathcal{M}$. It  follows from  Proposition \ref{lem_FL} that \eqref{eq_3.4bbb} holds with $r_0=r$, which, taking  (b) and \eqref{eq_3.3bbb}   into account,   entails    
$$r\le    r_0= \inf  ({\rm RP}_c)=\max ({\rm RD}_c) = \max_{(u,\lambda)\in \U\times S^+}    [-(\lambda G_u)^\ast(-c)].$$  
This means that there exists $(\bar u, \bar \lambda) \in  \U \times S^+$ such that $ (-c, -r_0 ) \in \epi (\bar \lambda G_{\bar u})^\ast$. Now, as    $r\le r_0 $, one has     $(-c,- r )\in  \epi (\bar \lambda G_{\bar u})^\ast $, and hence, $(c, r) \in           -\mathcal{M}_0$. We have proved that  $\oco\mathcal{M}_0\subset \mathcal{M}_0$ and the proof is complete. 
 \end{proof}

We now  provide some  sufficient  conditions for the convexity and closedness of  the robust moment cone $\mathcal{M}_0$. Assume from now to end this section that $\U$ is a subset of some topological vector space. The next result is a consequence of  
\cite[Propositions 5.1, 5.2]{DL18ACTA}.

\begin{proposition}\label{prop_conclo}
\label{pro_suffcond}   Assume that that $\U$ is a subset of some topological vector space and $\rm{ int }\, S \neq \emptyset$. Then   

{\rm (i)} If the collection $(u\mapsto G_u(x))_{x\in X}$ is uniformly $S^+$-concave, then $\mathcal{M}_0$ is convex. 

{\rm (ii)} If \  $\U$ is a compact set,  $Z$  is  a normed space, $u\mapsto G_u(x)$ is $S^+$-usc for all $x\in X$,  and the following Slater-type condition holds:  
$$(C_0)  \ \ \ \ \ \ \ \qquad  \qquad\forall u\in \U,\; \exists x_u\in X: G_u(x_u)\in-\inte S,  \ \ \ \ \ \ \ \qquad  \ \ \ \ \ \ \ \qquad $$
then $\mathcal{M}_{0}$ is  closed.
\end{proposition}

\begin{remark}\label{rem_2.1}
If\  \   $\U$ is a singleton then  it is easy to see that the assumption of Proposition  \ref{pro_suffcond}(i) automatically holds, and consequently, $\mathcal{M}_0$ is convex. Moreover, if the Slater  condition $(C_0)$ holds then   $\mathcal{M}_0$ is closed.
\end{remark}

\begin{remark} \label{rem24-nw} It is worth noticing that the Proposition \ref{prop_conclo}  and the next Corollary \ref{cor_3.1bis} 
constitute generalizations of Proposition 2 and Corollary 1 in \cite{GJLL}, respectively.  Propositions  \ref{convexity-N}-\ref{closedness-N} on the sufficient conditions for  the convexity and closedness of moment cones are of the same line of generalization which show the  role played by the Slater constraint qualification condition. 
\end{remark} 

\begin{corollary} [Sufficient condition for stable robust strong duality of $({\rm RP}_c)$]\label{cor_3.1bis}
Assume that the following conditions holds:

{\rm (i)} \ $\U$ is a  compact set, $Z$ is a normed space, 

{\rm (ii)} $(u\mapsto G_u(x))_{x\in X}$ is uniformly $S^+$-concave,

{\rm (iii)} $u\mapsto G_u(x)$ is $S^+$-usc for all $x\in X$,

{\rm (iv)} The Slater-condition $(C_0)$ holds.

\noindent
Then, the  stable robust strong duality holds for the pair {$({\rm RP}_c)-({\rm RD}_c)$}.
\end{corollary}

\begin{proof}
The conclusion follows  from Theorem \ref{thm_StrD} and Proposition  \ref{pro_suffcond}.
\end{proof}


 We  now consider a special  case of  $(\rm{RP}_c)$   where,  for each  $u\in \U$, $G_u$ is  an  affine mapping, say $\bar G_u$,   defined as 
$$\bar G_u(.):=A_u(.)-\omega_u,\quad \forall u\in \U, $$
where  $(A_u)_{u\in\U}\subset \L(X,Z)$ and $(\omega_u)_{u\in \U}\subset Z$. 
Let   $c\in X^\ast$. The problem $(\rm{RP}_c)$  becomes\footnote{The model of Problem  $({\rm RLP}_c) $ was considered  in \cite{DGLM17-Optim} where some characterizations of its solutions were proposed.} 
\begin{align}\label{problem_RLP}
({\rm RLP}_c)\qquad& {\inf  }\;  \la c,x\ra\\
\textrm{subject to  }  &A_u(x)\in \omega_u-S, \ \forall u\in \U, x \in  X.\notag
\end{align}

For each $(u,\lambda)\in \U\times S^+$, by a simple calculation, one gets
\begin{equation*}
\epi (\lambda \bar G_u)^\ast= (\lambda A_u, \la\lambda, b_u\ra)+\mathbb{R}_+(0_{X^\ast},1). 
\end{equation*}
Here we understand that $\lambda A_u $ is an element of $X^\ast$ with $(\lambda A_u)(x) = \la  \lambda , A_u(x) \ra $, for all $ x \in X$. 
So, the  $\mathcal{M}_0$ defined in \eqref{M0}  collapses to
\begin{equation}
\label{M_1}
\mathcal{M}_1:=  \{(\lambda A_u, \la \lambda, \omega_u\ra),\; (u,\lambda)\in\U\times S^+\} +\mathbb{R}_+ (0_{X^\ast}, 1), 
\end{equation}
and 
one has,  
\begin{equation*}
\inf_{x\in X} \Big\{\la c,x\ra +\lambda G_u(x)\Big\}=\begin{cases}
-\la \lambda,\omega_u\ra &\textrm{if } c=-\lambda A_u\\
-\infty &\textrm{otherwise.}
\end{cases}
\end{equation*}
The dual problem of $(\rm{RLP}_c)$,   specialized from $({\rm RD}_c)$,  turns to be 
\begin{align*}
	({\rm RLD}_c)\qquad& {\sup  }\;  -\la \lambda,\omega_u\ra \\
\textrm{subject to} \ \ &(u,\lambda)\in \U\times S^+,\; c=-\lambda A_u.
\end{align*}

\begin{corollary}[Characterization of stable robust strong duality for $(\rm{RLP}_c)$]
\label{thm_lStrD}
The following statements are equivalent:

{$\rm(a)$} $\mathcal{M}_1 $ is a closed and convex subset  of $X^*\times \mathbb{R}$.
 
{$\rm(b)$}  The  stable robust strong duality holds for the pair {$({\rm RLP}_c)-({\rm RLD}_c)$}, i.e.,
$$\inf ({\rm RLP}_c) = \max ({\rm RLD}_c),  \ \forall c \in X^\ast. $$ 
\end{corollary}

The previous corollary is a direct consequence of Theorem \ref{thm_StrD}. Moreover, apply Corollary \ref{cor_3.1bis} with $\bar G_u(.)=A_u(.)-\omega_u$ one also gets a sufficient condition for stable robust strong duality for $({\rm RLP}_c)$.

\section{Robust Linear Infinite Problem and Its Robust Duals }  
We retain the notations in Section 2 and let $c\in X^\ast$.    

\subsection{Statement of Robust Linear Infinite Problems and Their  Robust Duals} 
 Consider the {\it linear infinite programming}  with {\it  uncertain input-parameters} of the form:
\begin{align}
({\rm ULIP}_c)\quad &\inf\;  \langle c,x\rangle \notag\\
\textrm{subject to }\ \  &  \langle a_t,x \rangle \le b_t ,\;\forall t\in T,  x \in X, 
\end{align}
where   $(a_t , b_t)$  belongs to an uncertainty set $\U_t$ with $\emptyset\ne\U_t\subset X^*\times \mathbb{R}$ for all $t\in T$.

The {\it robust counterpart} of $({\rm ULIP}_c)$ is
\begin{align}
({\rm RLIP}_c)\quad &\inf\;  \langle c,x\rangle \notag\\
\textrm{subject to  }\ \  &  \langle x^\ast,x \rangle \le r,\;\forall (x^\ast,r)\in\U_t,\; \forall t\in T,  x \in X.
\end{align}
Assume  that the problem $({\rm RLIP}_c)$ is feasible for each $c \in X^\ast$, i.e., 
$$\F:=\{x\in X: \langle x^*,x \rangle \le r ,\;\forall (x^*,r)\in\U_t,\; \forall t\in T\}\ne \emptyset, \ \forall c \in X^\ast$$
and set
$$\mathscr{U}:=\prod_{t\in T}\U_t\qquad \textrm{and}\quad \mathscr{V}:=\bigcup_{t\in T}\U_t.$$
By convention, we write $v = (v^1, v^2) \in X^\ast\times \mathbb{R}$ and  $u = (u_t)_{t \in T} \in \mathscr{U}$, with $u_t = (u^1_t, u^2_t) \in \U_t$. For brevity, we also write: $u=(u^1_t, u^2_t)_{t\in T} \in  \mathscr{U}$ instead of  $u=((u^1_t, u^2_t))_{t\in T} \in  \mathscr{U}$.

\ms

The robust problem of the model  (RLIP$_c$) was considered in several earlier works such as  
 \cite{DGLM17-Optim},  \cite{GJLL} (where $X = \R^n$, i.e., a  robust semi-infinite linear problem), \cite{LJL11}  where $X$ is a Banach space, $T$ is finite, objective function is a convex function, and for each $t \in T$, $\U_t$ has a special form  (problem (SP), page 2335), and  in  \cite{DGLV17-Robust} with a bit more general on  constraint linear inequalities, concretely, for all $t\in T$, $(x^\ast,r)$ is a function defined on $\U_t$ instead of $(x^\ast, r)\in \U_t$.


We now propose variants of robust dual problems for  $({\rm RLIP}_c)$:
\begin{align}
({\rm RLID}_c^1)\quad &\sup\;  [-\lambda v^2]\notag\\
\textrm{s.t.}\ \  & v\in\mathscr{V},\;\lambda\ge 0,\; c=-\lambda v^1,\notag\\
({\rm RLID}_c^2)\quad  &\sup\; \left [-\sum_{u\in \supp \lambda}\lambda_u u^2_t\right]\notag\\
\textrm{s.t.}\ \  &  t\in T,\; \lambda \in \mathbb{R}^{(\mathscr{U})}_+,\; c=-\sum_{u\in \supp \lambda}\lambda_u u^1_t, \notag\\
({\rm RLID}_c^3)\quad  &\sup\; \left [-\sum_{t\in \supp \lambda}\lambda_t u^2_t\right]\notag\\
\textrm{s.t.}\ \  &  u\in\mathscr{U},\;\lambda \in \mathbb{R}^{(T)}_+,\; c=-\sum_{t\in \supp\lambda}\lambda_t u^1_t,\notag\\
({\rm RLID}_c^4)\quad  &\sup_{\lambda\ge 0,\; t\in T} \inf_{x\in X}\sup_{v\in\U_t}[\la c+\lambda v^1,x\ra -\lambda v^2],\notag\\
({\rm RLID}_c^5)\quad  &\sup_{\lambda\ge 0,\; u\in \mathscr{U}} \inf_{x\in X}\sup_{t\in T}[\la c+\lambda u^1_t,x\ra -\lambda u^2_t],\notag\\
({\rm RLID}_c^6)\quad &\sup\;  \left[-\sum_{v\in \supp \lambda}\lambda_v v^2\right]\notag\\
\textrm{s.t. }\ \  &   \lambda \in \mathbb{R}^{(\mathscr{V})}_+,\; c=-\sum_{v \in \supp\lambda}\lambda_{v} v^1,\notag\\
({\rm RLID}_c^7)\quad &\sup_{\lambda\ge 0} \inf_{x\in X}\sup_{v\in \mathscr{V}}[\la c+\lambda v^1,x\ra -\lambda v^2]. \notag
\end{align}

It is worth observing firstly  that  
$({\rm RLID}_c^3)$ and $({\rm RLID}_c^6)$   are  (ODP) and  (DRSP)  in \cite{GJLL}, respectively. These two classes are   also    special case of (OLD)   and (RLD) in \cite{JL2} (where the  constraint functions are affine) and of 
  (RLD$^O$) and  (RLD$^C$) in \cite{DGLV17-Robust}, respectively.

The {\it  ``robust strong  duality (and also, stable  robust strong  duality)  holds for the pair $({\rm RLIP}_c)-({\rm RLID}^i_c)$"}, $i = 1, 2, \ldots, 7$,    is understood  as in the Definition \ref{robust-dual}. 
Note that robust strong  duality holds for $({\rm RLIP}_c)-({\rm RLID}^3_c)$ is known as ``primal worst equals dual best problem'' with the attainment of dual problem  \cite{DGLV17-Robust},    \cite{GJLL}.

\subsection{Relationship Between The Values of Dual Problems and Weak Duality} 

In this subsection we will establish some relations between  the values of the dual problems $({\rm RLID}_c^i)$ to each other and the weak duality  to each of the dual pairs   $({\rm RLIP}_c) - ({\rm RLID}_c^i) $, $i = 1, 2, \cdots, 7$. 

\begin{proposition}\label{pro_3.1ff}
 One has
\begin{equation} \label{eq-prop31}
\sup ({\rm RLID}_c^1)\le 
\begin{array}{c}\sup ({\rm RLID}_c^2)\\\sup ({\rm RLID}_c^3)\end{array}
\le \sup ({\rm RLID}_c^6).
\end{equation}
\end{proposition}

\begin{proof}
Observe that, for $k=1,2,3,6$, it holds $\sup({\rm RLID}_c^k)=\sup E_k$ with
\begin{align}
E_1&:=\{\alpha: v\in \mathscr{V},\; \lambda\ge 0,\; (c,\alpha)=-\lambda v\},\\
E_2&:=\{\alpha: t\in T,\; \lambda \in \mathbb{R}^{(\mathscr{U})}_+,\; (c,\alpha)=-\sum_{u\in\supp\lambda}\lambda_u u_t\},\\
E_3&:=\{\alpha: u\in\mathscr{U},\;\lambda \in \mathbb{R}^{(T)}_+,\; (c,\alpha)=-\sum_{t\in \supp\lambda}\lambda_t u_t\},\\
E_6&:=\{\alpha:  \lambda \in \mathbb{R}^{(\mathscr{V})}_+,\; (c,\alpha)=-\sum_{v \in \supp\lambda}\lambda_{v} v\}.\label{eq_6.4eeee}
\end{align}
So, to prove \eqref{eq-prop31}, it suffices  to verify that $E_i\subset E_j$ for $(i,j)\in \{(1,2), (1,3), (2,6), (3.6)\}$.

$\bullet$ [$E_1\subset E_2$] Take $\bar \alpha\in E_1$. Then, there are $\bar v\in  \mathscr{V}$ and $\bar{\lambda}\ge 0$ such that $(c,\bar v)=-\bar \lambda \bar v$. Now, take $\bar t\in T$ and $\bar u\in \mathscr{U}$ such that $\bar u_{\bar t}=\bar v$.  Define $\bar\lambda \in \mathbb{R}^{(\mathscr{U})}_+$  by $\bar \lambda_{\bar u}=\bar \lambda$ and $\bar \lambda_u=0$ whenever $u\ne \bar u$. Then, it is easy to see that
$$-\sum_{u\in\supp\bar \lambda}\bar \lambda_u u_{\bar t}=-\bar \lambda_{\bar u} \bar u_{\bar t}=-\bar\lambda \bar v=(c,\bar \alpha), $$
yielding  $\bar \alpha \in E_2$.

$\bullet$ [$E_1\subset E_3$] Can be done by using  the same argument as in the proof of $E_1\subset E_2$, just replace $\bar \lambda\in  \mathbb{R}^{(\mathscr{U})}_+$ by $\bar \lambda\in \mathbb{R}^{(T)}_+$ such that $\bar\lambda_{\bar t}=\bar\lambda$ and $\bar\lambda_t=0$ for all $t\ne \bar t$.

$\bullet$ [$E_2\subset E_6$] Take $\bar \alpha\in E_2$. Then, there exists  $(\bar t, \bar \lambda)\in  T\times \mathbb{R}^{(\mathscr{U})}_+$ satisfying 
$$-\sum_{u\in \supp\bar\lambda}\bar \lambda_u  u_{\bar t} = (c,\bar\alpha).$$
Consider the set-valued mapping $\mathcal{K}\colon \mathscr{V}\rightrightarrows\mathscr{U}$ defined by 
$$\mathcal{K}(v):=\{u\in \supp\bar\lambda: u_{\bar t}=v\}.$$
It is easy to see that the decomposition $\supp \bar \lambda =\bigcup_{v\in \mathscr{V}} \mathcal{K}(v)$ holds. Moreover, as $\supp\bar\lambda$ is finite,  $\dom \mathcal{K}$ is also finite (where $\dom \mathcal{K}:=\{v\in \mathscr{V}: \mathcal{K}(v)\ne \emptyset\}$).
 Take $\hat\lambda\in \mathbb{R}^{(\mathscr{V})}_+$ such that $\hat \lambda_{v}=\sum_{u\in \mathcal{K}(v)}\bar \lambda_u$ if $v\in \dom \mathcal{K}$ and  $\hat \lambda_{v}=0$ if $v\notin \dom \mathcal{K}.$
Then, one has
$$-\sum_{v\in \supp \hat\lambda}\hat\lambda_v v=-\sum_{v\in \dom \mathcal{K}} \sum_{u\in\mathcal{K}(v)} \bar \lambda_u u_{\bar t}=-\sum_{u\in \supp\bar\lambda} \bar\lambda_u u_{\bar t}= (c,\bar \alpha), $$
yielding $\bar\alpha\in E_6$.

$\bullet$ [$E_3\subset E_6$] Similar to the proof of [$E_2\subset E_6$].
\end{proof}

\begin{proposition}\label{pro_3.2ff}
 One has
\begin{equation} \label{eq-prop32} 
\sup ({\rm RLID}_c^1)\le 
\begin{array}{c}\sup ({\rm RLID}_c^4)\\\sup ({\rm RLID}_c^5)\end{array}
\le \sup ({\rm RLID}_c^7).\end{equation} 
\end{proposition}

\begin{proof}
It is worth noting firstly that, for any non-empty  sets $Y_1$ and $  Y_2$, any function   $f\colon Y_1\times Y_2\to \mathbb{R}$, it always holds
\begin{equation}\label{eq_6.1eee}
\sup_{y_1\in Y_1}\inf_{y_2\in Y_2} f(y_1,y_2)\le \inf_{y_2\in Y_2}\sup_{y_1\in Y_1} f(y_1,y_2).
\end{equation}

By  a simple calculation, one easily gets
\begin{align*}
\sup({\rm RLID}_c^1)&=\sup_{\substack{\lambda\ge 0,\; v \in \mathscr{V}}}\inf _{x\in X}(\la c+\lambda v^1,x\ra-\lambda v^2)\\
&=\sup_{\lambda\ge 0,\; t\in T}\sup_{w\in \U_t}\inf _{x\in X}(\la c+\lambda w^1,x\ra-\lambda w^2)\\
&=\sup_{\lambda\ge 0,\; u\in \mathscr{U}}\sup_{t\in T} \inf_{x\in X}[\la c+\lambda u^1_t,x\ra -\lambda u^2_t]
\end{align*}
(as $\mathscr{V}=\bigcup_{t\in T}\U_t=\{u_t: u\in \mathscr{U},\; t\in T\}$).
So, according to \eqref{eq_6.1eee},

\begin{gather*}
\sup({\rm RLID}_c^1)\le\sup_{\lambda\ge 0,\; t\in T}\inf _{x\in X}\sup_{w\in \U_t}[\la c+\lambda w^1,x\ra-\lambda w^2]= \sup({\rm RLID}_c^4),\\
\sup({\rm RLID}_c^1)\le\sup_{\lambda\ge 0,\; u\in \mathscr{U}} \inf_{x\in X}\sup_{t\in T}[\la c+\lambda u^1_t,x\ra -\lambda u^2_t]= \sup({\rm RLID}_c^5).
\end{gather*}
The other desired inequalities in \eqref{eq-prop32}  follow  from  \eqref{eq_6.1eee} in  a similar way as above.
\end{proof}

The weak duality for  the   primal-dual  pairs of  problems $({\rm RLIP}_c^i) - ({\rm RLID}_c^i)$,  $i = 1, 2, \cdots, 7$, will be given in the next proposition.

\begin{proposition}[Weak duality] \label{weakduality1}
\label{pro_3.3gg}
 One has
\begin{equation} \label{eq-prop32}  
\begin{array}{c}\sup ({\rm RLID}_c^6)\\ \sup ({\rm RLID}_c^7) \\ 
\end{array}
\le \inf ({\rm RLIP}_c).\end{equation} 
Consequently,  $\sup ({\rm RLID}_c^i) 
\le \inf ({\rm RLIP}_c)$ for all $i =1, 2, \cdots, 7$. 
\end{proposition}
\begin{proof} 
$\bullet$  {\sl Proof of  $\sup ({\rm RLID}_c^6)\le \inf ({\rm RLIP}_c)$:} 
Take  $\bar\lambda \in \mathbb{R}^{(\mathscr{V})}_+$,   and $\bar x\in X$ such that  $c=-\sum_{v \in \supp\lambda}\bar\lambda_{v} v^1$ and 
\begin{equation}\label{eq_3.12gg}
\la v^1,\bar x\ra - v^2\le 0,\quad v\in\mathscr{V}. 
\end{equation}
Then it is  easy to see that
$-\sum_{v\in \supp\bar\lambda} v^2\le - \sum_{v\in \supp\bar\lambda}\la v^1,\bar x\ra=\la c,\bar x\ra.$ 
So, by the definitions of  $ ({\rm RLID}_c^6)$ one has 
$\sup ({\rm RLID}_c^6) \leq  \la c,\bar x\ra $
for any $\bar x \in X$ satisfying \eqref{eq_3.12gg}, which  yields  $\sup ({\rm RLID}_c^6)\le \inf ({\rm RLIP}_c)$.

 \medskip

$\bullet$  {\sl Proof  of \  $\sup ({\rm RLID}_c^7)\le \inf ({\rm RLIP}_c)$:}  Take $\bar \lambda\ge 0$ and $\bar x\in X$ such that \eqref{eq_3.12gg} holds. For all $v\in \mathscr{V}$, as \eqref{eq_3.12gg} holds, one has $\la c+\bar\lambda v^1, \bar x\ra -\bar\lambda v^2\le \la c,\bar x\ra$. This yields that $\sup_{v\in \mathscr{V}}[\la c+\bar\lambda v^1, \bar x\ra -\bar\lambda v^2]\le \la c,\bar x\ra$ which, in turn, amounts for 
$$\inf_{x\in X}\sup_{v\in \mathscr{V}}[\la c+\bar\lambda v^1,  x\ra -\bar\lambda v^2]\le \la c,\bar x\ra.$$
The conclusion  follows.  
\end{proof}

\section{ Robust  Stable Strong   Duality  for (RLIP$_c$)}

In this section,   we will establish variants of   stable  robust strong  duality results for (RLIP$_c$). Some of them  cover the ones in 
 \cite{GJLL},  \cite{JL2} and the others are new.

Let us introduce variants of {\it robust moment cones} of (RLIP$_c$): 
\begin{align*}
\mathcal{N}_1&:=\cone \mathscr{V}+  \mathbb{R}_+ (0_{X^*},1), &&\mathcal{N}_2:=\bigcup_{ t\in T}\co\cone [\U_t\cup\{ (0_{X^*},1)\}],\\
\mathcal{N}_3&:=\bigcup_{ u\in\mathscr{U}}\co\cone [u(T)\cup\{ (0_{X^*},1)\}], \ \ \ \ &&\mathcal{N}_4:=\bigcup_{t\in T}\cone\oco[\U_t+ \mathbb{R}_+( 0_{X^*},1)],\\
\mathcal{N}_5&:=\bigcup_{u\in \mathscr{U}}\cone\oco[u(T)+\mathbb{R}_+( 0_{X^*},1)], &&
\mathcal{N}_6:=\co\cone \left[\mathscr{V}\cup \{(0_{X^*},1)\}  \right],   \\
\mathcal{N}_7&:=\cone \oco\left[\mathscr{V}+\mathbb{R}_+ (0_{X^*},1)\right], 
\end{align*}
where $u(T):=\{u_t: t\in T\}$ if  $u\in\mathscr{U}$.

Observe that  $\mathcal{N}_3$ is $M_{\ell f}$ in \cite{DGLM17-Optim},  and   
$ \mathcal{N}_3$ and  $\mathcal{N}_6 $ were introduced in  \cite{GJLL} and  known as   ``robust moment cone" and 
``characteristic cone", respectively.

\begin{theorem} [$1^{st}$ characterization of stable robust strong duality for (RLIP$_c$)]
\label{thm_3.2}
\noindent    For $ i \in \{ 1, 2, \ldots, 5\} $, consider the following statements:

$({\rm c}_i)$ $\mathcal{N}_i $ is a closed and convex subset of $X^*\times \mathbb{R}$,

$({\rm d}_i)$  The stable  robust strong  duality holds for the pair {$({\rm RLIP}_c)-({\rm RLID}_c^i)$}.   

\noindent
Then, one has   $[({\rm c}_i)\Leftrightarrow ({\rm d}_i)]$  for all $i\in\{1,2,\ldots, 5\}$.  
\end{theorem}

\begin{proof}  $\bullet$  $[({\rm c}_1)\Leftrightarrow ({\rm d}_1)]$
 Set  $Z=\mathbb{R}$, $S=\mathbb{R}_+$, $\U=\mathscr{V}$,  $A_v=v^1$ and $\omega_v=v^2$ for all  $v = (v^1, v^2) \in\mathscr{V}$. Then, $({\rm RLIP}_c)$ has the form of $({\rm RLP}_c)$ in \eqref{problem_RLP}. 
In such a setting,  the robust moment cone $\mathcal{M}_1$ in  \eqref{M_1} reduces to 
\begin{align*}
\mathcal{M}_1
&=\{(\lambda A_v,\la\lambda, \omega_v\ra): v\in \mathscr{V},\; \lambda\ge 0\}+  \mathbb{R}_+ (0_{X^*},1)\\
&=\{\lambda v:  v\in \mathscr{V},\; \lambda\ge 0\}+  \mathbb{R}_+ (0_{X^*},1)\\
&=\cone \mathscr{V}+  \mathbb{R}_+ (0_{X^*},1)=\mathcal{N}_1. 
\end{align*}
It is easy to see  that  the robust dual problem $({\rm RLD}_c)$ of  the  resulting  robust problem   $(\rm {RLP}_c)$ 
now turns be exactly    $({\rm RLID}_c^1)$,  and so, 
the equivalence  $[({\rm c}_1)\Leftrightarrow ({\rm d}_1)]$  follows  directly from Corollary \ref{thm_lStrD}.
\medskip

$\bullet$ $[({\rm c}_2)\Leftrightarrow ({\rm d}_2)]$   Set $Z=\mathbb{R}^{\mathscr{U}}$, $S=\mathbb{R}^{\mathscr{U}}_+$ (and consequently, $Z^*=\mathbb{R}^{(\mathscr{U})}$ and $S^+=\mathbb{R}^{(\mathscr{U})}_+$), $\U=T$,   $A_t= (u^1_t)_{u\in \mathscr{U}}$ and $\omega_t=   (u^2_t)_{u\in \mathscr{U}}$ for all $t\in T$.  Then  the problem   $({\rm RLIP}_c)$   possesses  the form 
 $({\rm RLP}_c)$.     
In   this setting, the set $\mathcal{M}_1$  in  \eqref{M_1}  becomes 
\begin{align*}
\mathcal{M}_1&=\{(\lambda A_t,\la\lambda, \omega_t\ra): t\in T,\;  \lambda\in  \mathbb{R}^{(\mathscr{U})}_+\}+  \mathbb{R}_+ (0_{X^*},1)\\
&=\left\{\left(\sum_{u\in\supp\lambda}\lambda_u u^1_t,\sum_{u\in\supp\lambda}\lambda_u u^1_t\right): t\in T,\;  \lambda\in  \mathbb{R}^{(\mathscr{U})}_+\right\}+  \mathbb{R}_+ (0_{X^*},1)\\
&=\left\{\sum_{u\in\supp\lambda}\lambda_uu_t: t\in T,\;  \lambda\in  \mathbb{R}^{(\mathscr{U})}_+\right\}+  \mathbb{R}_+ (0_{X^*},1)\\
&=\left[\bigcup_{t\in T}\left\{\sum_{u\in\supp\lambda}\lambda_uu_t:  \lambda\in  \mathbb{R}^{(\mathscr{U})}_+\right\}\right]+  \mathbb{R}_+ (0_{X^*},1)\\
&= \left[\bigcup_{t\in T} \co\cone \U_t\right]+ \mathbb{R}_+ (0_{X^*},1)\quad \textrm{(note that $\{u_t: u\in \mathscr{U}\}=\U_t$)} \\
&= \bigcup_{t\in T} \left[\co\cone \U_t+ \mathbb{R}_+ (0_{X^*},1)\right]= \bigcup_{t\in T} \co\cone \left[ \U_t\cup \{(0_{X^*},1)\}\right]=\mathcal{N}_2, 
\end{align*}
and the dual problem of $({\rm RLD}_c)$ (in the new format)    has the form  $({\rm RLID}_c^2)$. 
The equivalence $[({\rm c}_2)\Leftrightarrow ({\rm d}_2)]$ then follows from  Corollary \ref{thm_lStrD}.

$\bullet$ $[({\rm c}_3)\Leftrightarrow ({\rm d}_3)]$   We transform $({\rm RLIP}_c)$ to $({\rm RLP}_c)$ by   setting: $Z=\mathbb{R}^{T}$, $S=\mathbb{R}^{T}_+$ (hence, $Z^*=\mathbb{R}^{(T)}$ and $S^+=\mathbb{R}^{(T)}_+$), $\U=\mathscr{U}$,  $A_u=(u^1_t)_{t\in T}$ and $\omega_u=(u^2_t)_{t\in T}$ for all $u \in \mathscr{U}$. Then, one has
\begin{align*}
\mathcal{M}_1&=\{(\lambda A_u,\la\lambda, \omega_u\ra): u\in \mathscr{U},\;  \lambda\in  \mathbb{R}^{(T)}_+\}+  \mathbb{R}_+ (0_{X^*},1)\\
&=\left\{\sum_{t\in\supp\lambda}\lambda_t u_t: u\in \mathscr{U},\;  \lambda\in  \mathbb{R}^{(T)}_+\right\}+  \mathbb{R}_+ (0_{X^*},1)\\
&= \left[\bigcup_{u \in \mathscr{U}} \co\cone u(T)\right]+  \mathbb{R}_+ (0_{X^*},1)\quad\textrm{(note that $\left\{u_t: t\in T \right\}=u(T)$)}\\
&= \bigcup_{u\in \mathscr{U}} \co\cone \left[ u(T)\cup \{(0_{X^*},1)\}\right]=\mathcal{N}_3.
\end{align*}
and dual problem $({\rm RLD}_c)$ of the resulting problem  $({\rm RLP}_c)$   is exactly  $({\rm RLID}_3)$. The desired equivalence follows from Corollary \ref{thm_lStrD}.
\medskip

$\bullet$ $[({\rm c}_4)\Leftrightarrow ({\rm d}_4)]$ We now consider another  way of  transforming  $({\rm RLIP}_c)$ to the form  $({\rm RP}_c)$ by letting  $Z=\mathbb{R}$, $S=\mathbb{R}_+$, $\U=T$, and $G_t\colon X\to \overline{\mathbb{R}}$ such that $G_t(x)= \sup_{v\in\U_t} [\la v^1, x\ra-v^2]$ for all $t\in T$. Then, one has  (see \eqref{M0}) 
\begin{align*}
\mathcal{M}_0&=\bigcup_{t\in T,\; \lambda\ge 0} \epi (\lambda G_t)^\ast =\bigcup_{t\in T,\; \lambda\ge 0} \lambda\epi (G_t)^\ast\\
&=\bigcup_{t\in T} \cone\epi (G_t)^\ast =\bigcup_{t\in T} \cone\epi\left[\sup_{ v\in\U_t} (\la v^1,.\ra-v^2)\right]^*\\
&=\bigcup_{t\in T} \cone\oco \bigcup_{ v\in\U_t}\epi\left( \la v^1,.\ra-v^2\right)^*
\end{align*}
(the last equalities follows from \cite[Lemma 2.2]{GNg08}). On the other hand, for each $t\in T$ and $v\in \U_t$, by simple calculation one gets
$\epi \left( \la v^1,.\ra-v^2\right)^*=v+\mathbb{R}_+ (0_{X^*},1).$
So, 
\begin{align*}
\mathcal{M}_0 =\bigcup_{t\in T} \cone\oco \bigcup_{ v\in\U_t}[v+\mathbb{R}_+ (0_{X^*},1)] 
=\bigcup_{t\in T} \cone\oco [\U_t+\mathbb{R}_+ (0_{X^*},1)] =\mathcal{N}_4.
\end{align*}
It is easy to see that the dual problem $({\rm RD}_c)$ of the resulting problem $({\rm RP}_c)$  is nothing else but $({\rm RLID}_c^4)$, and   
the equivalence $[({\rm c}_4)\Leftrightarrow ({\rm d}_4)]$ is a consequence of  Theorem \ref{thm_StrD}.

\medskip

$\bullet$ $[({\rm c}_5)\Leftrightarrow ({\rm d}_5)]$  Again, we transform $({\rm RLIP}_c)$ to $({\rm RP}_c)$ but by another setting:   $Z=\mathbb{R}$, $S=\mathbb{R}_+$,  $\U=\mathscr{U}$, and $G_u\colon X\to \overline{\mathbb{R}}$ such that $G_u(x)= \sup_{t\in T}[\la u^1_t, x\ra-u^2_t)]$  for all $u\in \mathscr{U}$. 
 Then, one has  (see \eqref{M0})   
\begin{align*}
\mathcal{M}_0&=\bigcup_{u\in \mathscr{U},\; \lambda\ge 0} \epi (\lambda G_u)^\ast 
=\bigcup_{u\in \mathscr{U}} \cone\epi (G_u)^\ast\\
&=\bigcup_{u\in \mathscr{U}} \cone\epi\left[\sup_{ t\in T} (\la u^1_t,.\ra-u^2_t)\right]^*=\bigcup_{u\in \mathscr{U}} \cone\oco \bigcup_{ t\in T}\epi (\la u^1_t,.\ra-u^2_t)^*\\
&=\bigcup_{u\in \mathscr{U}} \cone\oco \bigcup_{ t\in T}[u_t+\mathbb{R}_+ (0_{X^*},1)] =\bigcup_{u\in \mathscr{U}} \cone\oco [u(T)+\mathbb{R}_+ (0_{X^*},1)]=\mathcal{N}_5, 
\end{align*}
and the robust dual problem  $({\rm RD}_c)$  of the new  problem $({\rm RP}_c)$ is exactly  $({\rm RLID}_c^5)$. The desired equivalence again  follows from Theorem \ref{thm_StrD}.
\end{proof}

\begin{remark}
Theorem \ref{thm_3.2} with  $i=3$ is \cite[Theorem 2]{GJLL} while 
  $i=6$ ($i = 3$, resp.) is similar to \cite[Proposition 5.2(ii)]{DGLV17-Robust} with   $i=C$ ($i = O$, resp.). 
\end{remark}

\begin{theorem} [$2^{nd}$  characterization for stable robust  strong duality for (RLIP$_c$)]
\label{thm_3.2bis} For $ i = 6, 7$, 
consider the next statements: 

$({\rm c}_i)$ $\mathcal{N}_i $ is a closed subset of $X^*\times \mathbb{R}$,

$({\rm d}_i)$  The stable  robust strong  duality holds for the pair {$({\rm RLIP}_c)-({\rm RLID}_c^i)$}.   

\noindent
Then  $[({\rm c}_i) \Leftrightarrow ({\rm d}_i)]$ for  $ i = 6, 7$.   
\noindent
\end{theorem}

\begin{proof}

$\bullet$ $[({\rm c}_6)\Leftrightarrow ({\rm d}_6)]$
The robust problem  $({\rm RLIP}_c)$ turns to be  $({\rm RLP}_c)$ if we set   $Z=\mathbb{R}^{\mathscr{V}}$, $S=\mathbb{R}^{\mathscr{V}}_+$,  $\U$ to be  a singleton, $A=(v^1)_{v\in \mathscr{V}}$ and $\omega=(v^2)_{v\in \mathscr{V}}$. In such a setting,  one gets 
\begin{align*}
\mathcal{M}_1&=\{(\lambda A,\la\lambda, \omega\ra):  \lambda\in  \mathbb{R}^{(\mathscr{V})}_+\}+  \mathbb{R}_+ (0_{X^*},1)\\
&=\left\{\sum_{v\in\supp\lambda}\lambda_v v:  \lambda\in  \mathbb{R}^{(\mathscr{V})}_+\right\}+  \mathbb{R}_+ (0_{X^*},1)\\
&= \co\cone \mathscr{V}+  \mathbb{R}_+ (0_{X^*},1)=\co\cone \left[\mathscr{V}\cup \{(0_{X^*},1)\}  \right]=\mathcal{N}_6, 
\end{align*}
while the robust dual problem of the new problem   $({\rm RLP}_c)$ 
(i.e., $({\rm RLD}_c)$)  is  non other than $({\rm RLID}_c^6)$. 
The equivalence $[({\rm c}_6)\Leftrightarrow ({\rm d}_6)]$ now follows from Corollary \ref{thm_lStrD} and the fact that the robust  moment cone is always convex whenever $\U$ is a singleton (see Proposition \ref{prop_conclo} and Remark \ref{rem_2.1}).

$\bullet$  $[({\rm c}_7)\Leftrightarrow ({\rm d}_7)]$ Set  $Z=\mathbb{R}$, $S=\mathbb{R}$, $\U$ to be  a singleton, and $G= \sup_{v\in \mathscr{V}}(\la v^1, .\ra-v^2)$. The problem  $({\rm RLIP}_c)$ now becomes   $({\rm RP}_c)$. On the other hand, one has (see \eqref{M0}) 
\begin{align*}
\mathcal{M}_0&=\bigcup_{\lambda\ge 0} \epi (\lambda G)^\ast =\cone \epi (\lambda G)^\ast\\
&=\cone\epi\left[\sup_{v\in \mathscr{V}}(\la v^1, .\ra-v^2)\right]^* = \cone\oco \bigcup_{v\in \mathscr{V}}\epi (\la v^1,.\ra-v^2)^*\\
&=\cone\oco \bigcup_{v\in \mathscr{V}}[v+\mathbb{R}_+ (0_{X^*},1)] =\cone \oco\left[\mathscr{V}+\mathbb{R}_+ (0_{X^*},1)\right]=\mathcal{N}_7, 
\end{align*}
while the  dual problem of $({\rm RD}_c)$ of the new problem  $({\rm RP}_c)$   is  $({\rm RLID}_c^7)$. The equivalence $[({\rm c}_7)\Leftrightarrow ({\rm d}_7)]$ is a consequence of  Theorem \ref{thm_StrD}, Proposition \ref{prop_conclo} (see also Remark \ref{rem_2.1}).
\end{proof}

\begin{remark} \label{rem41a}There may have sone  more  ways of transforming (RLIP$_c$) to the form of  (RP$_c$)  which give rise to some more robust dual problems for (RLIP$_c$),  for instance,

$(\alpha)$ Set  $Z=\mathbb{R}^T$, $S=\mathbb{R}_+^T$, $\U$ to be  a singleton, and $G= \left(\sup_{ v\in\U_t} [\la v^1, .\ra-v^2]\right)_{t\in T}$. Then (RLIP$_c$) reduces to the form of  (RP$_c$)  with no uncertainty as now $\U$ is a singleton. In this setting,   the moment  
cone $\mathcal{M}_0$ becomes 
\begin{align}\label{eq_7.1ee}
\mathcal{M}_0=\bigcup_{\lambda\in \mathbb{R}_+^{(T)}}\epi \left[\sum_{t\in T} \lambda_t \sup_{ v\in\U_t} (\la v^1, .\ra-v^2) \right]^*=:\mathcal{N}_8, 
\end{align}
and the robust dual problems now collapses to 

\begin{align*}
({\rm RLID}_c^8)\quad &\sup_{\lambda\in\mathbb{R}_+^{(T)}} \inf_{x\in X}\left[\la c,x\ra +\sum_{t\in \supp\lambda}\lambda_t \sup_{v\in\U_t}\left(\la v^1,x\ra -v^2\right)\right].
\end{align*}  

$(\beta)$  Set $Z=\mathbb{R}^\mathscr{U}$, $S=\mathbb{R}^\mathscr{U}_+$, $\U$ to be  a singleton, and  $G= (\sup_{t\in T}[\la u^1_t, .\ra-u^2_t)])_{u\in \mathscr{U}}$.   Then, the problem (RLIP$_c$) turns to be of the model   (RP$_c$),  and  one has
\begin{align*}
\mathcal{M}_0
&=\co\cone \bigcup_{u \in \mathscr{U}}  \oco\left[u(T)+\mathbb{R}_+(0_{X^*},1)\right]=:\mathcal{N}_9.   
\end{align*}
The corresponding dual problem is 
\begin{align*}
({\rm RLID}_c^9)\quad &\sup_{\lambda \in\mathbb{R}_+^{(\mathscr{U})}} \inf_{x\in X}\left[\la c,x\ra +\sum_{u\in \supp\lambda}\lambda_u \sup_{t\in T}\left(\la u^1_t,x\ra -u_t^2\right)\right].\notag
\end{align*}

For the mentioned cases, we  get  also the relation between the values of these two dual problems: 
$$
            \sup({\rm RLID}_c^6)\le \sup({\rm RLID}_c^8)\ \ \textrm{ and  } \ \ \sup({\rm RLID}_c^6)\le \sup({\rm RLID}_c^9), 
   $$  
   and  weak duality hold as well: 
   $$\begin{array}{c}\sup ({\rm RLID}_c^8)\\ \sup ({\rm RLID}_c^9) \\ 
\end{array}
\le \inf ({\rm RLIP}_c).$$        
Moreover, under some suitable conditions, robust strong duality holds, similar to the ones in \cite[Proposition 5.2(ii)]{DGLV17-Robust}.     
\end{remark}

\begin{remark} \label {rem_4.3new} 
From Propositions \ref{pro_3.1ff}-\ref{pro_3.3gg} and Remark \ref{rem41a},      we get an overview of the relationship between the values of robust dual problems and weak duality of each pair of primal-dual problems which can be described  as in the next figure:  

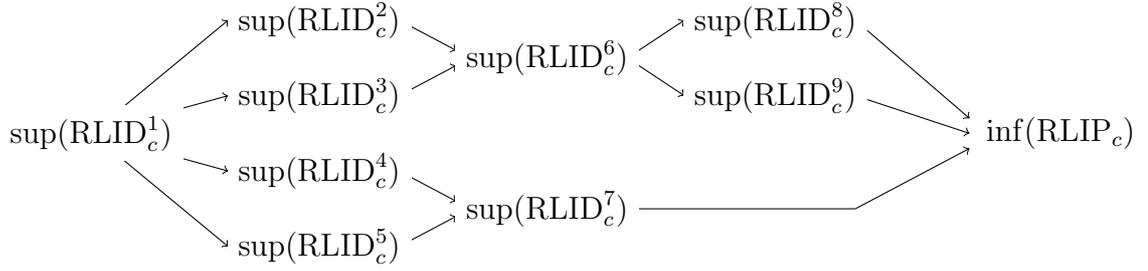
\begin{center}
\begin{tikzpicture}[x=1.5cm,y=1cm]
\draw[->] (-3,0) -- (-1.8,1.5); 
	\draw[->] (-0.3,1.5) -- (0.2,1.1); 
\draw[->] (-3,0) -- (-1.8,0.5); 
	\draw[->] (-0.3,0.5) -- (0.2,0.9); 
		\draw[->] (1.7,1) -- (2.2,1.5); 
			\draw[->] (3.7,1.5) -- (4.7,0.2); 
		\draw[->] (1.7,1) -- (2.2,0.5); 
			\draw[->] (3.7,0.5) -- (4.7,0); 
\draw[->] (-3,0) -- (-1.8,-0.5); 
	\draw[->] (-0.3,-0.5) -- (0.2,-0.9); 
\draw[->] (-3,0) -- (-1.8,-1.5); 
	\draw[->] (-0.3,-1.5) -- (0.2,-1.1); 
		\draw[-] (1,-1) -- (3.7,-1); \draw[->] (3.7,-1) -- (4.7,-0.2); 
\draw (-3,0) node [fill=white] {$\sup ({\rm RLID}_c^1)$};
\draw (-1,1.5) node [fill=white] {$\sup({\rm RLID}_c^2)$};
\draw (-1,0.5) node [fill=white] {$\sup({\rm RLID}_c^3)$};
\draw (-1,-0.5) node [fill=white] {$\sup({\rm RLID}_c^4)$};
\draw (-1,-1.5) node [fill=white] {$\sup({\rm RLID}_c^5)$};
\draw (1,1) node [fill=white] {$\sup({\rm RLID}_c^6)$};
\draw (1,-1) node [fill=white] {$\sup({\rm RLID}_c^7)$};
\draw (3,1.5) node [fill=white] {$\sup({\rm RLID}_c^8)$};
\draw (3,0.5) node [fill=white] {$\sup({\rm RLID}_c^9)$};
\draw (5.5,0) node [fill=white] {$\inf({\rm RLIP}_c)$};
\end{tikzpicture}
\end{center}
\noindent where by $a\longrightarrow b$  we mean  $a\le b$.
\end{remark}


As we have seen from  the previous theorems and from  the previous section, the closedness and convexity of robust moment cones play crucial roles in closing the dual gaps for the primal-dual pairs of robust  problems.  In the left of this section,   we will give some sufficient conditions for  the mentioned  properties of these  cones whose proofs are rather long and will be  put in  the last section: Appendices.

\begin{proposition}[Convexity of moment cones] \label{convexity-N}  The next assertions hold: 
$\phantom{x}$
\begin{itemize}[nosep]
\item[$\rm(i)$] If $\mathscr{V}$ is convex  then $\mathcal{N}_1$ is convex, 


\item[$\rm(ii)$] If $\{x^\ast\in X^*: (x^\ast,r)\in \U_t\}$ is convex for all $t\in T$ then $\mathcal{N}_3$ is convex, 

\item[$\rm(iii)$] Assume that $T$ is a  convex subset of some vector space, and that, for all $t\in T$, $\U_t=\U^1_t\times\U^2_t$ with $\U^1_t\subset X^*$ and $\U^2_t\subset \mathbb{R}$. Assume further that, for each $t\in T$ and $x\in X$, the function $t\mapsto \sup_{x^\ast\in \U^1_t} \la x^\ast, x\ra$ is affine and the function $t\mapsto \inf \U^2_t$ is convex. Then, $\mathcal{N}_4$ is convex, 

\item[$\rm(iv)$] The sets $\mathcal{N}_6$,   $\mathcal{N}_7$   are  convex\footnote{$\mathcal{N}_8$,  $\mathcal{N}_9$ are also convex.}. 
\end{itemize}
\end{proposition}

\begin{proof} See Appendix A. \end{proof}

\begin{proposition}[Closedness of moment cones]\label{closedness-N}  The following assertions are true.
$\phantom{x}$
\begin{itemize}[nosep]
\item[$\rm(i)$] If $ \mathscr{V}$ is compact and 
\begin{equation}
\label{eq_4.2zab}
\forall v\in \mathscr{V},\; \exists \bar x\in X: \la v^1,\bar x\ra<v^2, 
\end{equation}
then $\mathcal{N}_1$ is closed.

\item[$\rm(ii)$]  If $T$ is compact, $t\mapsto \sup_{v\in \U_t} [\langle v^1, x\rangle-v^2]$ is usc for all $x\in X$, and
\begin{equation}
\label{eq_4.3zab}
\forall t\in T ,\; \exists x_t\in X: \sup_{v\in \U_t} [\langle v^1, x_t\rangle-v^2]<0, 
\end{equation}
then $\mathcal{N}_4$ is closed.

\item[$\rm(iii)$]  If $\U_t$ is compact for all $t\in T$, $u\mapsto \sup_{t\in T}[\la u_t^1, x\ra - u_t^2]$ is usc for all  $x\in X$, and 
\begin{equation}
\label{eq_4.4zab}
\forall u \in \mathscr{U},\; \exists x_u\in X: \sup_{t\in T}[\la u_t^1, x_u\ra - u_t^2]<0, 
\end{equation}
then  $\mathcal{N}_5$ is closed.

\item[$\rm(iv)$]  If the following condition holds
\begin{equation}
\label{eq_4.5zab}
 \exists x\in X:  \sup_{v\in  \mathscr{V}} [\la v^1, x\ra - v^2]<0, 
\end{equation}
then $\mathcal{N}_7$ is closed.
\end{itemize}
\end{proposition}

\begin{proof} See Appendix B. \end{proof}

\section{Farkas-Type Results for  Infinite Linear systems with Uncertainty }

We retain the notations used in Sections 2, 3, and 4. 

 
   Let $ c \in X^\ast$, $T$ be an index set (possibly infinite), and let $\U_t $ be uncertainty set for each $t \in T$.   Consider the robust linear system of the form 
 \begin{align}
\ &  \langle x^\ast,x \rangle \le r,\qquad\forall (x^\ast,r)\in\U_t,\;  \forall t\in T,
\end{align}
which is the constraint system of the problem $(\mathrm{RLIP}_c)$ considered in Section 4. 

Based on the stable strong robust duality   results established in Section 4, we can derive the next robust Farkas-type results for the linear systems with uncertainty parameters (for a short survey on Farkas-type results, see, e.g., \cite{DJ-Top}).

\begin{corollary}[Robust Farkas lemma for linear system I]\label{cor_6.1ff}  Let $\emptyset \ne  \mathscr{V}   \subset X^\ast \times \R$.  
The following statements are equivalent:

\noindent
$({\rm i})$ For all $(c,s )\in X^\ast\times \mathbb{R}$ such that $\inf({\rm RLIP}_c) > - \infty$,  $x \in X$, the  next assertions are equivalent:

$(\alpha)$  $\langle x^\ast,x \rangle \le r,\;\forall (x^\ast,r)\in \mathscr{V} \; \Longrightarrow\; \la c,x\ra \ge s$,

$(\beta)$ $\exists (\bar x^\ast, \bar r) \in\mathscr{V},\; \exists \bar\lambda\ge 0 : 
\begin{cases}
\bar\lambda \bar x^\ast=-c,\\
\bar \lambda \bar r\le - s,
\end{cases}
$

\noindent 
$({\rm ii})$ $\cone \mathscr{V}+  \mathbb{R}_+ (0_{X^*},1)$ is convex and closed.
\end{corollary}

\begin{proof}
Take $(c,s)\in X^\ast\times \mathbb{R}$. Set $\Lambda := \{ (x^*, r , \lambda) :  (x^*, r ) \in \mathscr{V},  \lambda \in \R_+, \lambda x^* = -c \}$ 
and  $\Phi (x^*, r, \lambda) = - \lambda r$ for all   $(\bar x^*, \bar r, \bar \lambda) \in \Lambda $.  So, $\sup ({\rm RLID}_c^1) = \sup\limits_{(x^*, r , \lambda) \in \Lambda}  \Phi (x^*, r, \lambda)$.  
From the statements of the problems $({\rm RLIP}_c) $ and $({\rm RLID}_c^1) $, one has   
  \begin{eqnarray} 
  (\alpha)       &\Longleftrightarrow&   \inf ({\rm RLIP}_c)\ge s, \label{eqalpha}  \\
  (\beta)    &\Longleftrightarrow& \Big( \exists(\bar x^*, \bar r, \bar \lambda) \in \Lambda:   \sup ({\rm RLID}_c^1)\ge    \Phi (\bar x^*, \bar r, \bar \lambda) = -\bar \lambda \bar r       \geq s\Big)    \label{eqbeta}. 
  \end{eqnarray}
$\bullet$ $[(ii)  \Longrightarrow  (i)]$ Assume that $(ii) $ holds. Then it follows from Theorem \ref{thm_3.2} (with $i=1$) that 
\begin{eqnarray} \label{eqFarkas1}
(ii)   &\Longleftrightarrow&  \Big(\textrm{the stable  robust strong  duality holds for the pair } ({\rm RLIP}_c)-({\rm RLID}_c^1) \Big)  \notag\\
 &\Longleftrightarrow& \Big( \forall c \in X^\ast, \ \ \inf({\rm RLIP}_c) = \max ({\rm RLID}_c^1) \Big). 
\end{eqnarray} 
So, for  $c \in X^*$,   if $(\alpha)$ holds then  $\inf ({\rm RLIP}_c)\ge s$ and hence, we get from \eqref{eqFarkas1}, 
$$ \inf ({\rm RLIP}_c)  = \max  ({\rm RLID}_c) = \Phi (\bar x^*, \bar r, \bar \lambda)  = -\bar \lambda \bar r  \geq s,
$$ for some $ (\bar x^*, \bar r, \bar \lambda) \in \Lambda$, which means that $(\beta)$ holds, and so $[(\alpha) \Longrightarrow (\beta)]$. 

Conversely, if $(\beta)$ holds, then from \eqref{eqbeta} and the weak duality of the primal-dual pair $({\rm RLIP}_c)-({\rm RLID}_c^1)$, one gets the existence of $(\bar x^*, \bar r, \bar \lambda) \in \Lambda $ such that 
$$  
   \inf ({\rm RLIP}_c)  \geq     \sup ({\rm RLID}_c^1)\ge    \Phi (\bar x^*, \bar r, \bar \lambda) = -\bar \lambda \bar r       \geq s,    
$$
yielding $(\alpha)$. So, $[(\beta) \Longrightarrow (\alpha)]$ and consequently, we have proved that  $[(ii)  \Longrightarrow  (i)]$.

$\bullet$ $[(i)  \Longrightarrow  (ii)]$  Assume that $(i)$ holds.  Take   $s = \inf ({\rm RLIP}_c)  \in \R$    and  $c \in X^*$.   Then $(\alpha)$ holds and as $(i)$ holds, $(\beta)$ holds as well. This, together with the weak duality,  yields, for some   $ (\bar x^*, \bar r, \bar \lambda) \in \Lambda$  (see \eqref{eqbeta}),  
  $$  
  \inf ({\rm RLIP}_c)  \geq    \sup ({\rm RLID}_c^1) =   \Phi (\bar x^*, \bar r, \bar \lambda) = -\bar \lambda \bar r     \geq s = \inf ({\rm RLIP}_c), 
$$ 
meaning that the robust dual problem $({\rm RLID}_c^1)$ attains and $ \inf ({\rm RLIP}_c)  =    \max ({\rm RLID}_c^1) $. 
Since $c\in X^*$ is arbitrary, 
the stable  robust strong  duality holds for the pair  $({\rm RLIP}_c)-({\rm RLID}_c^1) $.  The fulfillment of $(ii)$ now follows from  Theorem \ref{thm_3.2} (with $i=1$).   
\end{proof}

\begin{remark}
Assume that  $\mathscr{V}$ is a convex and compact subset  of $X^\ast\times \mathbb{R}$ and that the Slater-type condition \eqref{eq_4.2zab} holds. According to Propositions \ref{convexity-N} and \ref{closedness-N}, one has $\mathcal{N}_1:= \cone \mathscr{V}+  \mathbb{R}_+ (0_{X^*},1)$ is closed and convex. So, it follows from Corollary \ref{cor_6.1ff}, $(\alpha)$ and $(\beta)$ in Corollary \ref{cor_6.1ff} are equivalent. This observation may apply to some of the next corollaries. 
\end{remark}

The next versions of robust Farkas lemmas follows from the same way as Corollary \ref{cor_6.1ff}, using  Theorem \ref{thm_3.2} with $i=2, 3$, and $i= 4$. 

\begin{corollary}[Robust Farkas lemma for linear system II]\label{cor_6.2ff}
The following statements are equivalent:

\noindent
$({\rm i})$ For all $(c,s)\in X^\ast\times \mathbb{R}$  such that $\inf({\rm RLIP}_c) > - \infty$,  $x \in X$,  the    next assertions are equivalent:

$(\alpha)$  $\langle x^\ast,x \rangle \le r,\;\forall (x^\ast,r)\in \mathscr{V}\; \Longrightarrow\; \la c,x\ra \ge s$,

$(\gamma)$ $ \exists \bar t\in T,\; \exists \bar \lambda \in \mathbb{R}^{(\mathscr{U})}_+ : 
\begin{cases}
\sum\limits_{u\in \supp \lambda}\bar\lambda_u u^1_{\bar t}=-c,\\
\sum\limits_{u\in \supp \lambda}\bar\lambda_u u^2_{\bar t}\le -s,
\end{cases}
$

\noindent 
$({\rm ii})$ $\bigcup_{ t\in T}\co\cone [\U_t\cup\{ (0_{X^*},1)\}]$ is convex and closed.
\end{corollary}

\begin{corollary}[Robust Farkas lemma for linear system III] {\rm \cite[Theorem 5.6]{DMVV17-Robust-SIAM}, \cite[Corollary 3]{GJLL},\cite[Theorem 6.1(i)]{DGLM17-Optim}}
\label{cor_6.3ff}
The following statements are equivalent:

\noindent
$({\rm i})$ For all $(c,s)\in X^\ast\times \mathbb{R}$,    such that $\inf({\rm RLIP}_c) > - \infty$,  $x \in X$,   the    next assertions are equivalent:

$(\alpha)$  $\langle x^\ast,x \rangle \le r,\;\forall (x^\ast,r)\in \mathscr{V}\; \Longrightarrow\; \la c,x\ra \ge s$,

$(\delta)$  $\exists \bar u\in\mathscr{U},\;\exists \bar \lambda \in \mathbb{R}^{(T)}_+ : 
\begin{cases}
\sum\limits_{t\in \supp\lambda}\bar\lambda_t \bar{u}^1_t=-c,\\
\sum\limits_{t\in \supp\lambda}\bar\lambda_t \bar{u}^2_t\le -s,
\end{cases}
$

\noindent 
$({\rm ii})$ $\bigcup_{ u\in\mathscr{U}}\co\cone [u(T)\cup\{ (0_{X^*},1)\}]$ is convex and closed, 
where $u(T):=\{u_t: t\in T\}$ for all $u\in\mathscr{U}$.
\end{corollary}

\begin{corollary} [Robust Farkas lemma for linear system IV]     \label{cor_6.4ff}
The following statements are equivalent:

\noindent
$({\rm i})$ For all $(c,s)\in X^\ast\times \mathbb{R}$,   such that $\inf({\rm RLIP}_c) > - \infty$,  $x \in X$,   the   next assertions are equivalent:

$(\alpha)$  $\langle x^\ast,x \rangle \le r,\;\forall (x^\ast,r)\in \mathscr{V}\; \Longrightarrow\; \la c,x\ra \ge s$,

$(\epsilon)$   $\exists \bar t\in T,\; \exists \bar\lambda\ge 0$ such that $\forall x\in X, \;\forall \varepsilon>0,\; \exists (x_0^\ast, r_0)\in \U_{\bar t}$ satisfying
$$ \la c+\bar\lambda x_0^\ast,x \ra -\bar\lambda r_0  \ge s - \varepsilon ,$$

\noindent 
$({\rm ii})$ $\bigcup_{ u\in\mathscr{U}}\co\cone [u(T)\cup\{ (0_{X^*},1)\}]$ is convex and closed.
\end{corollary}

\begin{remark} 
It worth noting that robust Farkas-type results can be established in the same way as in the previous corollaries, corresponding to the 
stable robust strong  duality for pairs $({\rm RLIP}_c) - ({\rm RLID}_c^j)$ with $j=5,\ldots,  9$.
 The results corresponding to $i=6$ can be considered as a version of \cite[Corollary 4]{GJLL} with $\mathscr{V}$ replacing $\gph \mathscr{U}$.
\end{remark}

\section{Linear Infinite  Problems   with Sub-affine Constraints}

The results in previous sections for robust linear infinite problems $ ({\rm RLIP}_c)$ ($c \in X^*$) can be extended to a rather broader class of robust problems by a similar approaching. Here we consider a concrete class of problems:  The robust linear problems with sub-affine constraints.  

Denote by $\mathscr{P}_0(X^\ast)$ the set of all the nonempty, $w^{\ast }-$%
closed convex subsets of $X^{\ast }$.  Let 
 $T$ be a possibly infinite index set, $(\U_t)_{t\in T}  \subset \mathscr{P}_0(X^\ast)\times \mathbb{R}$   be a collection of nonempty uncertainty sets.  
We introduce the sets 
$$\mathfrak{V}:=\bigcup_{t\in T}\U_t \quad \textrm{and}\quad \mathfrak{U}=\prod_{t\in T}\U_t.$$
By convention, for each $V \in\mathscr{P}_0(X^\ast)\times\mathbb{R}$, we write $V = (V^1, V^2)$ with  $V^1 \subset X^\ast$ and $V^2 \in \R$.  In  some case, we also considered  $V = (V^1, V^2) \in\mathscr{P}_0(X^\ast)\times\mathbb{R} $ as a subset of the set $X^\ast\times \mathbb{R}$  ideybntifying  $V $ with  $V^1\times \{V_2\}$. In the same way, for $U\in \mathfrak{U}$, we write  $U=(U_t)_{t\in T}$ with $U_t  = (U_t^1, U_t^2) \in \U_t$ for each $t \in T$.

For each $c\in X^\ast$, consider the robust linear problem with sub-affine constraints: 
\begin{align*}
({\rm RSAP}_c)\qquad &\inf\;  \langle c,x\rangle \notag\\
\textrm{subject to }\ \ \ \  & \sigma_{\A_t}(x)\le b_t,\;    \forall (\A_t,b_t)\in\U_t,\; \forall t\in T,  x \in X. 
\end{align*}
Here $\sigma_{\A_t}$ denotes the support function of the set $\A_t \subset X^\ast$, i.e., $\sigma_{\A_t}(x) := \sup\limits_{x^* \in \A_t} \la x^*, x\ra$.

We now introduce  two robust dual problems for ${\rm (RSAP}_c)$:
\begin{align*}
({\rm RSAD}_c^1)\qquad &\inf\;  - \lambda v^2 \notag\\
\textrm{subject to  } \ \ \ \ &  V\in \mathfrak{V},\; v = (v^1, v^2) \in V, \\ 
 &\lambda\ge 0, \; c=-\lambda v^1.  \notag
\end{align*}

\begin{align*}
\ \ \ \ \ \ \ \ \ ({\rm RSAD}_c^2)\qquad &\inf\;  - \sum_{U\in \supp\lambda} \lambda_Uv^2_U \notag\\
\textrm{subject to } \ \ \ & (v_U)_{U\in \mathfrak{U}}\in (U_t)_{U\in\mathfrak{U}}, \; c=-\sum_{U\in \supp\lambda} \lambda_Uv^1_U \\
 &\ t \in T, \ \  \lambda\in \mathbb{R}_+^{(\mathfrak{U})}. 
 \end{align*}

We can state stable robust strong duality for ${\rm (RSAP}_c)$ as follows:

\begin{corollary}[Stable robust strong duality for  $ ({\rm RSAP}_c)$\   I]
The following statements are equivalent:

$\rm(g_1)$ $\mathcal{R}_1:=\cone \mathfrak{U}+\mathbb{R}_+(0_{X^\ast},1) $ is a closed and convex subset  of $X^*\times \mathbb{R}$, 
 
$\rm(h_1)$  The  stable robust strong duality holds for the pair {$({\rm RSAP}_c)-({\rm RSAD}_c^1)$}, i.e.,  
$$\inf ({\rm RSAP}_c)=\max ({\rm RSAD}_c^1),\quad\forall c\in X^\ast.$$
\end{corollary}

\begin{proof} 
Set  $Z=\mathbb{R}$, $S=\mathbb{R}^+$, $\U=\mathfrak{V}$ and $G_V(.)=\sigma_{V^1}(.)-V^2$ for all $V   = (V^1, V^2) \in \mathfrak{V}$.  Then the 
problem ${\rm (RSAP}_c)$ possesses the form  of $({\rm RP}_c)$.   
The corresponding robust moment cone $\mathcal{M}_0$ now becomes 
\begin{align*}
\mathcal{M}_0&=\bigcup_{(V,\lambda)\in \mathfrak{V}\times\mathbb{R}_+}\epi (\lambda G_V)^\ast
=\bigcup_{(V,\lambda)\in \mathfrak{V}\times\mathbb{R}_+}\lambda\epi ( G_V)^\ast\\
&=\cone\bigcup_{V\in\mathfrak{V}}\epi ( G_V)^\ast
=\cone\bigcup_{V\in \mathfrak{V}}\epi (\sigma_{V^1}(.)-V^2)^\ast\\
&=\cone\bigcup_{V\in \mathfrak{V}}[ V^1 \times\{V^2\} +\mathbb{R}_+(0_{X^\ast},1)]
=\cone [\mathfrak{U}+\mathbb{R}_+(0_{X^\ast},1)]\\
&=\cone \mathfrak{U}+\mathbb{R}_+(0_{X^\ast},1)=\mathcal{R}_1
\end{align*}
and the dual problem $({\rm RD}_c)$ of the resulting problem  $({\rm RP}_c)$  is exactly the problem $({\rm RSAD}_c^1)$.
The conclusion now  follows  from Theorem \ref{thm_StrD}.\end{proof}

\begin{corollary}[Stable robust strong duality  for $ ({\rm RSAP}_c)$  \ II] Assume that  for all   $V = (V^1, V^2) \in\mathfrak{V}$,    $V^1$ is a $w^*$-compact subset of $X^\ast$.   
Then the  following statements are equivalent:

$\rm(g_2)$ $\mathcal{R}_2:= \bigcup_{t\in T} \co\cone \left[ \U_t\cup \{(0_{X^*},1)\}\right] $ is a closed and convex subset  of $X^*\times \mathbb{R}$, 
 
$\rm(h_2)$  The  stable robust strong duality holds for the pair {$({\rm RSAP}_c)-({\rm RSAD}_c^2)$}.
\end{corollary}

\begin{proof} 

Under the current assumption, $\sigma_{V^1}$ is continuous on $X$ for all $V = (V^1, V^2) \in\mathfrak{V}$.
Take $Z=\mathbb{R}^{\mathfrak{U}}$, $S=\mathbb{R}^{\mathfrak{U}}_+$, $\U=T$,   $G_t=(\sigma_{U^1_t}(.)-U^2_t)_{U\in \mathfrak{U}}$
 for all $t\in T$. Then the problem $({\rm RSAP}_c)$  turns to the  model    $({\rm RP}_c)$ and in   this setting, the moment cone $\mathcal{M}_0 $ becomes: 
\begin{align*}
\mathcal{M}_0&=\bigcup_{(t,\lambda)\in T\times \mathbb{R}_+^{(\mathfrak{U})}}\epi (\lambda G_t)^\ast
=\bigcup_{(t,\lambda)\in T\times \mathbb{R}_+^{(\mathfrak{U})}}\epi \left(\sum_{U\in \supp\lambda}\lambda_U \sigma_{U^1_t}(.)-U^2_t\right)^\ast\\
&=\!\!\!\!\!\! \bigcup_{(t,\lambda)\in T\times \mathbb{R}_+^{(\mathfrak{U})}}\!\!\sum_{U\in \supp\lambda}\!\!\!\!\lambda_U \epi \left(\sigma_{U^1_t}(.)-U^2_t\right)^\ast \!\!\!\!  =\!\!\!\!\bigcup_{(t,\lambda)\in T\times \mathbb{R}_+^{(\mathfrak{U})}}\!\!\sum_{U\in \supp\lambda}\!\!\!\!\lambda_U  \left[  U_t^1\times \{U_t^2\} +    \mathbb{R}_+ (0_{X^*},1)  \right]\\
&=\bigcup_{(t,\lambda)\in T\times \mathbb{R}_+^{(\mathfrak{U})}}\sum_{U\in \supp\lambda}\lambda_U\left[  U_t +    \mathbb{R}_+ (0_{X^*},1)  \right]
= \bigcup_{t\in T} \co\cone \left[ \U_t\cup \{(0_{X^*},1)\}\right]=\mathcal{R}_2. 
\end{align*}
Moreover, the dual problem of the new problem  $({\rm RP}_c)$ turns to be exactly  $({\rm RSAD}_c^2)$.
The conclusion now  follows from Theorem \ref{thm_StrD}.\end{proof}



From the above results on the (stable) robust  strong duality, one can use the same argument as the one in  Section 5 to get some versions of (stable) robust Farkas lemma for systems involved sub-affine functions with uncertain parameters.   Concretely, we can state the following robust Farkas lemmas for this class of systems and omit the proofs.

\begin{corollary}\label{corolRLAP1}
The following statements are equivalent:

\noindent
$({\rm i})$ For all $(c,s)\in X^\ast\times \mathbb{R}$, next assertions are equivalent:

$(\alpha'')$  $\sigma_{\A_t}(x)\le b_t,\;    \forall (\A_t,b_t)\in\U_t,\; \forall t\in T \, \Longrightarrow \la c, x\ra \geq s$. 

$(\beta'')$  $\exists \bar V\in \mathfrak{V},\; \exists(\bar x^\ast, \bar r)\in \bar V,\; \exists\bar\lambda\ge 0 :
\begin{cases}
 \bar\lambda \bar x^\ast=-c\\
\bar\lambda \bar r\le -s.
\end{cases}
$

\noindent 
$({\rm ii})$  $\cone \mathfrak{U}+\mathbb{R}_+(0_{X^\ast},1)$  is a closed and convex subset of $X^*\times \mathbb{R}$. 

\end{corollary}

\begin{corollary}
The following statements are equivalent:

\noindent
$({\rm i})$ For all $(c,s)\in X^\ast\times \mathbb{R}$, next assertions are equivalent:

$(\alpha'')$  $\sigma_{\A_t}(x)\le b_t,\;    \forall (\A_t,b_t)\in\U_t,\; \forall t\in T \, \Longrightarrow \la c, x\ra \geq s$. 

$(\gamma'')$  $\exists  \bar t \in T,\; \exists (\bar x^\ast_U, \bar r_U)_{U\in \mathfrak{U}}\in (U_{\bar t})_{U\in\mathfrak{U}},\;  \bar \lambda\in \mathbb{R}_+^{(\mathfrak{U})} :\\
\null \ \ \ \ \ \ \ \ \sum\limits_{U\in \supp\bar\lambda} \bar\lambda_U\bar x^\ast_U=-c, \ \ \textrm{and} \ \ 
\sum\limits_{U\in \supp\bar\lambda} \bar\lambda_U\bar r_U\le -s.
$

\noindent 
$({\rm ii})$  $\bigcup_{t\in T} \co\cone \left[ \U_t\cup \{(0_{X^*},1)\}\right]$  is a closed and convex subset of $X^*\times \mathbb{R}$. 

\end{corollary}

\noindent  {\it Duality for Linear Infinite Programming Problems.}   We now consider a special case of (RLIP$_c$):  the {\it linear infinite programming problems. } 
\begin{align}
({\rm LIP}_c)\quad &\inf\;  \langle c,x\rangle \notag\\
\textrm{s.t. }\ \  &x\in X,\;\langle a_t,x \rangle \le b_t,\;\forall t\in T,\notag
\end{align}
where $T$ is an arbitrary (possible infinite) index set, $c\in X^*$, $a_t \in X^*$, and $b_t \in \R$ for all $t \in T$.  In the case where $X =\mathbb{R}^n$ this problem is often known as linear semi-infinite problem  (see \cite{GL98} and also, \cite{CKL}, \cite{DW} for applications of this model in finance). 
x
We consider  $({\rm LIP}_c)$ in a new look:  a special case of $({\rm RLIP}_c)$ where all uncertainty sets $\U_t$, $t\in T$, are singletons for all $t \in T$, say,  $\U_t = \{(a_t,b_t)\}$, and then   $\mathscr{U} = \prod_{t\in T}\U_t  $ is also a singleton, say $
\mathscr{U} = \left\{ \Big( (a_t, b_t)\Big) _{ t \in T} \right\}$,   while  $\mathscr{V}=\{(a_t,b_t): t\in T\}$. 
We now have:

$\bullet$ All the three "robust" dual problems $({\rm RLID}_c^1)$, $({\rm RLID}_c^2)$, $({\rm RLID}_c^4)$ of the  problem $({ \rm LIP}_c)$ (considered as $({ \rm RLIP}_c)$)    collapse to  
\begin{align}
\hskip-1.5cm ({\rm LID}_c^1)\quad& \sup\; [ -\lambda b_t]\notag\\
\textrm{subject to  }\ \ \  & t\in T,\;\lambda\ge 0,\; c=-\lambda a_t, \notag
\end{align}
and in this situation, the three corresponding moments cones $\mathcal{N}_1$,   $\mathcal{N}_2$, and $\mathcal{N}_4$ reduce to  the moment cone   corresponding to the pair $({\rm LIP}_c) - ({\rm LID}_c^1)$: 
    $$ \mathcal{E}_1 :=\bigcup_{t\in T}\co\cone\{(a_t, b_t), (0_{X^*},1)\}.   $$

$\bullet$  All the three  "robust" dual problems $({\rm RLID}_c^3)$, $({\rm RLID}_c^6)$, $({\rm RLID}_c^8)$ of the new-formulated  problem $({\rm RLIP}_c)$ collapse to the next problem (which  is introduced in \cite{GL98} for the case where $X = \R^n$)
\begin{align}
({\rm LID}_c^2)\quad &\sup\;  \left[-\sum_{t\in \supp\lambda} \lambda_t b_t\right]\notag\\
\textrm{subject to  } \ \  &\lambda\in \mathbb{R}^{(T)}_+ ,\;  c=- \sum_{t\in \supp\lambda}\lambda_t a_t, \notag
\end{align}
and, in the same way as   above,  the three corresponding moments cones $\mathcal{N}_3$,   $\mathcal{N}_6$, and $\mathcal{N}_8$ reduce to moment cone corresponding to the pair   $({\rm LIP}_c) - ({\rm LID}_c^2):$ 
$$ \mathcal{E}_2:=\co\cone\{(a_t, b_t),t\in T; (0_{X^*},1)\} .$$

$\bullet$ All the three  dual problems $({\rm RLID}_c^5)$, $({\rm RLID}_c^7)$, $({\rm RLID}_c^9)$ of the resulting problem $({\rm RLIP}_c)$ reduce to:
$$\ \ \ \ \ ({\rm LID}_c^3)\qquad\sup_{\lambda\ge 0}  \inf_{x\in X} \sup_{t\in T}\Big[\la c,x\ra+\la \lambda a_t,x\ra -\lambda b_t\Big], $$
while the three robust moment cones $\mathcal{N}_5$,   $\mathcal{N}_7$, and $\mathcal{N}_9$ all reduce to the moment cone corresponding to  the pair  $({\rm LIP}_c) - ({\rm LID}_c^3):$ 
$$  \mathcal{E}_3    :=\cone\oco\{(a_t, b_t),t\in T; (0_{X^*},1)\}. $$

Moreover, for all $c\in X^\ast$,     one has   (see  Remark \ref{rem_4.3new}),  
\begin{equation}\label{eq_6.3bis}
\sup ({\rm LID}_c^1)\le 
\begin{array}{c}\sup ({\rm LID}_c^2)\\\sup ({\rm LID}_c^3)\end{array}
\le \inf ({\rm LIP}_c).
\end{equation}

As consequences of Theorems \ref{thm_3.2} and \ref{thm_3.2bis}, we have

\begin{corollary} [Principles of stable robust strong duality for $ {\rm(LIP}_c)$]\label{thm_2.1nww}
The following assertions are true.

\noindent ${\rm (i)} $ The next two statements are equivalent:

\hskip.5cm $\rm(e_1)$ $\mathcal{E}_1 $ is a closed and convex subset of $X^*\times \mathbb{R}$,

\hskip.5cm $\rm(f_1)$ The stable robust strong  duality holds for the pair $({\rm LIP}_c)-({\rm LID}^1_c)$, i.e., \\
\null \ \ \ \ \ \ \ \ \ $\inf({\rm LIP}_c)=\max ({\rm LID}_c^1)$ for all $c\in X^*$.
\smallskip

\noindent
${\rm (ii)} $   For each $i=2,3$, the following  statements are equivalent:

\hskip.5cm $\rm(e_i)$ $\mathcal{E}_i $ is a closed subset of $X^*\times \mathbb{R}$.

\hskip.5cm $\rm(f_i)$ The stable robust strong  duality holds for the pair $({\rm LIP}_c)-({\rm LID}^i_c)$.
\end{corollary}

\begin{remark}  It is clear that in this setting, one can specify  sufficient conditions   in Propositions \ref{convexity-N} and \ref{closedness-N} to guarantee  the convexity and closedness of the moment cones $\mathcal {E}_i$, $i = 1, 2, 3$,   and hence,  the stable robust strong  duality  for the pair $({\rm LIP}_c)-({\rm LID}^i_c)$ for $i = 1, 2, 3$ hold as well.  
\end{remark}

\noindent {\it Farkas-Type Results for  Linear Infinite  Systems.}    
Similar to what is done in the Section 5,  the duality results of the primal-dual pairs of problems $({\rm LIP}_c)-({\rm LID}^j_c)$,   $j = 1, 2, 3$ will give rise to some new  variants of generalized Farkas lemmas for linear infinite systems.  By this way,  it is easy to see that for the case  $j=2$ we will  get  a version  of Farkas lemma which goes  back to \cite[Corollary 3.1.2]{GL98} in the case where $X = \R^n$. 
In the next corollaries, we realize the process for $j=1$ and $j=3$, and to the best of our knowledge,  these resulting  versions of Farkas lemmas for linear infinite systems obtained here are new. Their  proofs are similar to those of Corollaries \ref{cor_6.1ff}-\ref{cor_6.4ff} and will be omitted.

\begin{corollary}[Farkas lemma for linear infinite systems I]      \label{corolLIP1}
The following statements are equivalent:

\noindent
$({\rm i})$ For all $(c,s)\in X^\ast\times \mathbb{R}$, next assertions are equivalent:

\hskip.5cm $(\alpha')$ $\langle a_t,x \rangle \le b_t,\;\forall t\in T \; \Longrightarrow\; \la c,x\ra \ge s$,

\hskip.5cm $(\beta')$ $ \exists \bar t\in T,\; \exists \bar \lambda \ge 0: 
\bar\lambda a_{\bar t}=-c\; \ \  \textrm{and} \; \ 
\bar\lambda b_{\bar t} \le -s,$

\noindent 
$({\rm ii})$  $\bigcup_{t\in T}\co\cone\{(a_t, b_t), (0_{X^*},1)\}$ is a closed and convex subset of $X^*\times \mathbb{R}$. 

\end{corollary}

\begin{corollary} [Farkas lemma for linear infinite systems II]            \label{corolLIP3}
The following statements are equivalent:

\noindent
$({\rm i})$ For all $(c,s)\in X^\ast\times \mathbb{R}$, next assertions are equivalent:

\hskip.5cm $(\alpha')$ $\langle a_t,x \rangle \le b_t,\;\forall t\in T \; \Longrightarrow\; \la c,x\ra \ge S$,

\hskip.5cm $(\delta')$ $\exists \bar\lambda\ge 0: \big[\forall x\in X, \;\forall \varepsilon>0,\; \exists t_0\in T: \la c+\bar\lambda a_{t_0},x \ra -\bar\lambda b_{t_0}+\varepsilon\ge s\big],$

\noindent 
$({\rm ii})$  $\cone\oco\{(a_t, b_t),t\in T; (0_{X^*},1)\}$ is a closed subset of $X^*\times \mathbb{R}$. 
\end{corollary}

\vskip-.5cm 
\section{Appendices} 
{\bf Appendix A.} {\it Proof of Proposition  \ref{convexity-N}. } 
 (i)  From the proof of Theorem \ref{thm_3.2} for the case $i=1$, we can see that the problem $({\rm RLIP}_c)$ can be transformed to $({\rm RP}_c)$ with $Z=\mathbb{R}$, $S=\mathbb{R}_+$, $\U=\mathscr{V}$, $G_v(.)= v^1(.)-v^2$ for all $v\in \mathscr{V}$, and in such a case, $\mathcal{M}_0=\mathcal{N}_1$.
Observe  that the  functions $v\mapsto \la v^1,x\ra -v_2$, $x\in X$, are concave (actually, they are affine). Together with the the fact that   $\mathscr{V}$ is convex and  $Z=\mathbb{R}$,  the collection $(v\mapsto G_v(x))_{x\in X}$ is  {uniformly $\mathbb{R}_+$-concave}. So, in the light of Proposition  \ref{prop_conclo}(i), $\mathcal{M}_0$ is convex, and  so is $\mathcal{N}_1$.  

\noindent
{ (ii)} By the same argument as above, to prove $\mathcal{N}_3$ is convex, it is sufficient to show that   the collection $(u\mapsto G_u(x))_{x\in X}$ is  {uniformly $\mathbb{R}_+^{(T)}$-concave} with $\U=\mathscr{U}$, $Z=\mathbb{R}^T$ and $G_u(.)=(\la u^1_t,.\ra-u^2_t)_{t\in T}$ for all $u\in \mathscr{U}$ (the setting in the proof of Theorem \ref{thm_3.2} for the case $i=3$).
Now, take arbitrarily $\lambda,\mu\in \mathbb{R}^{(T)}_+$ and  $u,w\in \mathscr{U}$. Let $\bar \lambda\in  \mathbb{R}^{(T)}_+$ and $\bar u\in \mathscr{U}$ such that $\bar \lambda_t=\lambda_t+\mu_t$, $\bar u^2_t=\min\{u^2_t, w^2_t\}$ and
$$\bar u_t^1=\begin{cases}
\frac{1}{\bar \lambda_t}(\lambda_t u^1_t+\mu_t w^1_t),&\textrm{if } \lambda_t+\mu_t\ne 0\\
u^1_t,&\textrm{otherwise}
\end{cases}$$
($\bar u \in\mathscr{U}$ as $\{x^\ast\in X^*: (x^\ast, r)\in \U_t\}$ is convex for all $t\in T$). Then, it is easy to check that
$$\lambda_t(\la u^1_t,x\ra -u^2_t)+\mu_t(\la w^1_t,x\ra -w^2_t)\le\bar \lambda_t (\la \bar u^1_t,x\ra-\bar u^2_t),\quad\forall t\in T,\; \forall x\in X,$$
and consequently, 
\begin{equation*}
\sum_{t\in T}\lambda^1_t(\la u^1_t,x\ra -u^2_t)+\sum_{t\in T}\lambda^2_t(\la w^1_t,x\ra -w^2_t)\le\sum_{t\in T}\bar \lambda_t (\la \bar u^1_t,x\ra-\bar u^2_t),\quad \forall x\in X, 
\end{equation*}
which means   $\lambda G_{u}(x)+\mu G_{w}(x)\le \bar\lambda G_{\bar u}(x)$ for all $x\in X$, yielding   the  uniform $\mathbb{R}_+^{(T)}$-concavity of the      collection $(u\mapsto G_u(x))_{x\in X}$.  The conclusion now follows from Proposition  \ref{prop_conclo}(i).

\noindent
{ (iii)} Recall that  $\mathcal{N}_4$ is a specific form of $\mathcal{M}_0$ with $Z=\mathbb{R}$, $S=\mathbb{R}_+$, $\U=T$, and $G_t(.)= \sup_{v\in\U_t} [\la v^1, .\ra-v^2]$ for all $t\in T$  (the setting in the proof of Theorem \ref{thm_3.2} for the case $i=4$).
Now, for each $t\in T$ and $x\in X$, as  $\U_t=\U^1_t\times\U^2_t$ (with $\U^1_t\subset X^*$ and $\U^2_t\subset \mathbb{R}$), it holds 
$$G_t(x)=\sup_{x^\ast\in \U^1_t} \langle x^\ast,x\rangle -\inf_{r\in \U^2_t}r=\sup_{x^\ast\in \U^1_t} \langle x^\ast,x\rangle- \inf \U^2_t.$$
So, for all $x\in X$, because $T$ is convex,   $t\mapsto \sup_{x^\ast\in \U^1_t} \la x^\ast, x\ra$ is affine, and $t\mapsto \inf \U^2_t$ is convex, the function $t\mapsto G_t(x)$ is concave.
This accounts for  the    uniform $\mathbb{R}^{(T)}_+$-concavity of the     collection $(t\mapsto G_t(x))_{x\in X}$. The conclusion again follows from Proposition  \ref{prop_conclo}(i).

{ (iv)} Consider the ways of transforming $({\rm RLIP}_c)$ to $({\rm RP}_c)$ in the proofs of Theorem \ref{thm_3.2bis}  for the case $i=6,7$.  Note  that, in these ways, the uncertain set $\U$ is always a singleton. So, the corresponding qualifying sets (i.e, $\mathcal{N}_6$ and  $\mathcal{N}_7$)  are always convex   (see Remark \ref{rem_2.1}).  \qquad \hfill $\square$

{\bf Appendix B.} {\it Proof of Proposition \ref{closedness-N}. }
Recall that $\mathcal{N}_i$, $i=1,2,\ldots,7$, are  specific forms of $\mathcal{M}_0$ following the corresponding ways transforming of $({\rm RLIP}_c)$ to $({\rm RP}_c)$ considered in the proofs of Theorems \ref{thm_3.2} and   \ref{thm_3.2bis}.   So, to prove that $\mathcal{N}_i$ is closed, we make use of  Proposition   \ref{prop_conclo}(ii), which  provides some  sufficient condition for the closedness of the robust moment cone $\mathcal{M}_0$.

\noindent
{ (i)}      For $i = 1$, let us consider the way of transforming $({\rm RLIP}_c)$ to $({\rm RP}_c)$ by  setting $Z=\mathbb{R}$, $S=\mathbb{R}_+$, $\U=\mathscr{V}$, and $G_v(.)= \la v^1, .\ra-v^2$ for all $v \in\mathscr{V}$.  
For all $x\in X$, it is easy to see that the function $v\mapsto G_v(x)=\la v^1,x\ra -v^2$ is continuous, and hence, it is $\mathbb{R}^+$-usc (see Remark \ref{rem_2.1eeee}(iii)). Moreover,   $\gph\mathscr{U}$ is compact, $\mathbb{R}$ is normed space, and \eqref{eq_4.2zab} ensures the fulfilling of condition $(C_0)$ in Proposition   \ref{prop_conclo}. The closedness of $\mathcal{N}_1$ follows from  Proposition   \ref{prop_conclo}(ii).

\noindent
{ (ii)} For $i = 4$, consider the way of transforming with the setting $Z=\mathbb{R}$, $S=\mathbb{R}_+$, $\U=T$, and $G_t(.)= \sup_{v\in\U_t} [\la v^1, .\ra-v^2]$ for all $t\in T$. 
One has that $\U=T$ is a compact set, that $t\mapsto G_t(x)= \sup_{v\in \U_t} [\langle v^1, x\rangle-v^2]$ is usc and hence, it  is $\mathbb{R}^+$-usc, and that Slater-type condition $(C_0)$ holds (as \eqref{eq_4.3zab} holds). The conclusion now follows from  Proposition   \ref{prop_conclo}(ii).

\noindent 
{ (iii)} Consider the way of transforming which corresponds to $i=5$, i.e.,  we consider  $Z=\mathbb{R}$, $S=\mathbb{R}_+$, $\U=\mathscr{U}$, and  $G_u(.)= \sup_{t\in T}[\la u^1_t, .\ra-u^2_t]$  for all $u\in \mathscr{U}$.  As $\mathscr{U}=\prod_{t\in T} \U_t$,  the assumption that $\U_t$ is compact for all $t\in T$ which entails the compactness of  $\mathscr{U}$. The other assumptions ensure the fulfillment of conditions  in  Proposition   \ref{prop_conclo}(ii) and the conclusion follows  from this very proposition.   

\noindent 
{ (iv)}  For $ i = 7$, using  the same argument as above in   transforming  $({\rm RLIP}_c)$ to $({\rm RP}_c)$    in the proof of Theorems  \ref{thm_3.2bis}. As by  this way,   the uncertainty set is a singleton, and hence, $\mathcal{N}_7$ is convex   (see Remark \ref{rem_2.1}). Now from Proposition   \ref{prop_conclo}(ii),    
Slater-type condition  ensures the closedness of the robust moment cone  $\mathcal{N}_7$, as desired. 
  \qquad \hfill $\square$

\end{document}